\date{} 
\title{Liouville quantum gravity and the Brownian map III:\\
 the conformal structure is determined}
\author{Jason Miller and Scott Sheffield}
\documentclass[12pt,naturalnames]{article}
\usepackage{amsmath}
\usepackage{amssymb}
\usepackage{amsthm}
\usepackage{amsfonts}
\usepackage{graphicx}
\usepackage{color}
\usepackage{tabularx}
\usepackage{subfig}
\usepackage{enumerate}
\usepackage{comment}
\usepackage{microtype}

\def\@rst #1 #2other{#1}
\newcommand\MR[1]{\relax\ifhmode\unskip\spacefactor3000 \space\fi
  \MRhref{\expandafter\@rst #1 other}{#1}}
\newcommand{\MRhref}[2]{\href{http://www.ams.org/mathscinet-getitem?mr=#1}{MR#2}}

\usepackage[margin=1.19in]{geometry}

\usepackage[usenames,dvipsnames]{xcolor}
\usepackage[pdftitle={Liouville quantum gravity and the Brownian map III: the conformal structure is determined},
  pdfauthor={Jason Miller and Scott Sheffield},
colorlinks=true,linkcolor=NavyBlue,urlcolor=RoyalBlue,citecolor=PineGreen,bookmarks=true,bookmarksopen=true,bookmarksopenlevel=2,unicode=true,linktocpage]{hyperref}

\newcommand{\CW}{{\mathcal W}}
\newcommand{\CF}{{\mathcal F}}
\newcommand{\CS}{{\mathcal S}}

\allowdisplaybreaks

\newif\ifhyper\IfFileExists{hyperref.sty}{\hypertrue}{\hyperfalse}

\ifhyper\usepackage{hyperref}\fi

\newif\ifdraft
\drafttrue
\numberwithin{equation}{section}
\numberwithin{figure}{section}

\newtheorem{theorem}{Theorem}
\numberwithin{theorem}{section}

\newtheorem{lemma}[theorem]{Lemma}
\newtheorem{proposition}[theorem]{Proposition}

\theoremstyle{remark}
\theoremstyle{remark}

\newcommand{\R}{\mathbf{R}}
\renewcommand{\C}{\mathbf{C}}
\newcommand{\D}{\mathbf{D}}

\newcommand{\N}{\mathbf{N}}

\definecolor{purple}{rgb}{0.7,0,0.7}
\definecolor{gray}{rgb}{0.6,0.6,0.6}
\definecolor{dgreen}{rgb}{0.0,0.4,0.0}
\definecolor{dblue}{rgb}{0.0,0.0,0.5}

\newcommand{\SH}{{\rm sh}}
\newcommand{\dsphere}{d}

\newcommand{\ol}{\overline}
\newcommand{\ul}{\underline}
\newcommand{\wh}{\widehat}
\newcommand{\wt}{\widetilde}

\newcommand{\CA}{{\mathcal A}}
\newcommand{\CG}{{\mathcal G}}

\newcommand{\CB}{{\mathcal B}}

\newcommand{\CT}{{\mathcal T}}
\newcommand{\CBM}{{\mathcal S}_{{\mathrm {BM}}}}

\newcommand{\s}{{\mathbf S}}
\newcommand{\giv}{\,|\,}

\newcommand{\one}{{\mathbf 1}}

\def\diam{\mathop{\mathrm{diam}}}

\def\dist{\mathop{\mathrm{dist}}}

\newcommand{\SLE}{{\rm SLE}}
\newcommand{\QLE}{{\rm QLE}}

\newcommand{\CQ}{{\mathcal Q}}

\newcommand{\CM}{{\mathcal M}}
\newcommand{\CU}{{\mathcal U}}

\newcommand{\p}{{\mathbf P}}
\newcommand{\pr}[1]{\p\!\left[#1\right]}
\newcommand{\ex}[1]{\E\!\left[#1\right]}
\newcommand{\var}[1]{\mathrm{var}\!\left(#1\right)}
\newcommand{\cov}[2]{\mathrm{cov}\!\left(#1,#2\right)}

\def\Ito/{It\^o}

\def \P {{\bf P}}

\def \p {{\P}}
\def \E {{\bf E}}

\newcommand{\qdist}{d_\CQ}
\newcommand{\oqdist}{\ol{d}_\CQ}

\newcommand{\qdistT}{\wt{d}_\CQ}
\newcommand{\ball}[2]{B(#1,#2)}
\newcommand{\qball}[2]{B_\CQ(#1,#2)}

\newcommand{\qhull}[2]{B_\CQ^\bullet(#1,#2)}
\newcommand{\qhullj}[3]{B_\CQ^{\bullet,#3}(#1,#2)}
\newcommand{\qhullT}[2]{\wt{B}_\CQ^\bullet(#1,#2)}
\newcommand{\qhullTj}[3]{\wt{B}_\CQ^{\bullet,#3}(#1,#2)}
\newcommand{\qdiam}{\mathrm{diam}_\CQ}

\begin{document} \maketitle

\begin{abstract}
Previous works in this series have shown that an instance of a $\sqrt{8/3}$-Liouville quantum gravity (LQG) sphere has a well-defined distance function, and that the resulting metric measure space (mm-space) agrees in law with the Brownian map (TBM). In this work, we show that given {\em just} the mm-space structure, one can a.s.\ recover the LQG sphere. This implies that there is a canonical way to parameterize an instance of TBM by the Euclidean sphere (up to M\"obius transformation). In other words, an instance of TBM has a canonical conformal structure.

The conclusion is that TBM and the $\sqrt{8/3}$-LQG sphere are equivalent. They ultimately encode the same structure (a topological sphere with a measure, a metric, {\em and} a conformal structure) and have the same law. From this point of view, the fact that the conformal structure a.s.\ determines the metric and vice-versa can be understood as a property of this unified law.  The results of this work also imply that the analogous facts hold for Brownian and $\sqrt{8/3}$-LQG surfaces with other topologies.
\end{abstract}
\newpage
\tableofcontents
\newpage

\parindent 0 pt
\setlength{\parskip}{0.25cm plus1mm minus1mm}

\medbreak {\noindent\bf Acknowledgements.} 
We have benefited from conversations about this work with many people, a partial list of whom includes Omer Angel, Itai Benjamini, Nicolas Curien, Hugo Duminil-Copin, Amir Dembo, Bertrand Duplantier, Ewain Gwynne, Nina Holden, Jean-Fran{\c{c}}ois Le Gall, Gregory Miermont, R\'emi Rhodes, Steffen Rohde, Oded Schramm, Stanislav Smirnov, Xin Sun, Vincent Vargas, Menglu Wang, Samuel Watson, Wendelin Werner, David Wilson, and Hao Wu.

We thank Ewain Gwynne, Nina Holden, and Xin Sun for comments on an earlier version of this article.  We also thank an anonymous referee for a very careful reading of this article, which led to many improvements in the exposition.

We would also like to thank the Isaac Newton Institute (INI) for Mathematical Sciences, Cambridge, for its support and hospitality during the program on Random Geometry where part of this work was completed.  J.M.'s work was also partially supported by DMS-1204894 and J.M.\ thanks Institut Henri Poincar\'e for support as a holder of the Poincar\'e chair, during which part of this work was completed.  S.S.'s work was also partially supported by DMS-1209044, DMS-1712862, a fellowship from the Simons Foundation, and EPSRC grants {EP/L018896/1} and {EP/I03372X/1}.

\section{Introduction}
\label{sec::intro}

\subsection{Overview}
\label{sec::overview}

The unit area Brownian map (TBM) and the unit area Liouville quantum gravity (LQG) sphere (with parameter $\gamma = \sqrt{8/3}$) are two natural continuum models for ``random surfaces'' which are both homeomorphic to the Euclidean sphere $\s^2$. Both objects come naturally endowed with an area measure. However, an instance of TBM additionally comes endowed with a metric space structure (i.e., a two-point distance function) while an instance of the LQG sphere comes endowed with a conformal structure (i.e., it can be canonically parameterized by $\s^2$ up to M\"obius transformation).

This article is the third and last in a series of papers that explain how to endow each of these random objects with the {\it other's} structure in a canonical way and show that once this is done the two objects agree in law. In other words, once both TBM and the $\sqrt{8/3}$-LQG sphere are endowed with both structures, they encode exactly the same information and have exactly the same law. From this point of view, any theorem about TBM is henceforth also a theorem about the $\sqrt{8/3}$-LQG sphere and vice-versa.

The first two papers in the current series \cite{qlebm, qle_continuity} show how to equip an instance~$\CS$ of the unit area $\sqrt{8/3}$-LQG sphere (as defined in \cite{dms2014mating,quantum_spheres}) with a metric in such a way that the resulting metric measure space (mm-space) agrees in law with TBM.

To phrase this another way, suppose one first samples an instance $\CS$ of the LQG sphere, and then uses it to construct an instance $\CBM$ of TBM (by endowing $\CS$ with a metric structure and then forgetting about the conformal structure) together with the identity homeomorphism $\phi$ from $\CS$ to $\mathcal \CBM$.  Then the results of \cite{qlebm, qle_continuity} imply that the law of $(\CS, \CBM)$ is a {\em coupling} of the $\sqrt{8/3}$-LQG sphere and TBM.  Moreover, it is shown in \cite{qlebm,qle_continuity} that in this coupling the first object a.s.\ determines the second.  The current paper shows that the second object a.s.\ determines the first; indeed, given only~$\CBM$ one can a.s.\ reconstruct both~$\CS$ and~$\phi$. In particular, this implies that there is a.s.\ a canonical way to endow an instance of TBM with a conformal structure (namely, the conformal structure it inherits from $\mathcal S$).

Although the three papers belong to the same series, let us stress that the methods and results are extremely different. The current paper implements a subtle conformal removability analysis that has no parallel in either \cite{qlebm} or \cite{qle_continuity}, the first two papers in the series. As explained in Section~\ref{subsec::outline}, the current paper {\em does} have some strategic similarity to \cite[Section~10]{dms2014mating} but the conceptual details that appear here are more difficult and delicate. Let us also stress the practical importance of the current paper. Once one knows that an instance of TBM comes with a {\em canonical} conformal structure, one can construct Brownian motions, conformal loop ensembles, and other natural random geometric objects {\em on top of} an instance of TBM.  This allows one to formulate many scaling limit conjectures (e.g., that simple random walk on random planar maps scales to Brownian motion on TBM) and may pave the way to a deeper understanding of the relationships between various discrete and continuous models. Moreover, conformal symmetries and conformal field theory have obviously been of great importance to the physics literature on random surfaces. Now that we have endowed TBM with a conformal structure, one can apply these techniques to the study of TBM.

Throughout most of this paper, we will actually work with an infinite volume variant of TBM called the {\em Brownian plane} (TBP), as defined in \cite{cl2012brownianplane}, which was shown in \cite{qle_continuity} to be equivalent to an infinite volume variant of the LQG-sphere, called a {\em quantum cone}, as defined in \cite{SHE_WELD, dms2014mating}. We will then use these results about infinite volume surfaces to deduce analogous results about finite volume spheres and disks.

We remark that, in some sense, our conclusions are even more general.  First we observe that in a certain sense both length space structure and conformal structure are ``local'' properties. To be more precise, let $U_1, U_2, \ldots, U_k$ be any finite cover of the sphere with open neighborhoods.  It is not hard to see that if one knows the conformal structure on each $U_i$, one can recover the conformal structure of the whole sphere. Similarly, if one knows the ``internal metric'' structure of each $U_i$ (i.e., one can define the length of any path that lies strictly inside one of the $U_i$) then this determines the length of any path (and hence the overall metric) on the entire sphere, because if $\eta:[0,1] \to \CBM$ is such a path, then the connected components of the sets $\eta^{-1}(U_i)$ form an open cover of $[0,1]$ and hence must have a finite subcover since $[0,1]$ is compact.)

The method for producing a length space structure on LQG, as described in \cite{qlebm, qle_continuity}, was shown to be ``local'' in the sense that it is a.s.\ the case that the internal metric on any open subset $U_i$ of a $\sqrt{8/3}$-LQG sphere instance is determined (in a particular concrete way) by the measure and conformal structure {\em within} $U_i$. This paper will show that the method for recovering the conformal structure from the metric is local in a similar sense.

Informally, these statements imply that we can impose a canonical length space structure on {\em any} surface that looks like a $\sqrt{8/3}$-LQG sphere locally, and a canonical conformal structure on {\em any} surface that looks locally like an instance of TBM. This should allow one to address, for example, higher genus variants of TBM and LQG.

However, let us stress if one has a {\em specific} method for generating a higher genus form of TBM --- and another {\em specific} method for generating a higher genus LQG surface, one would still have to do some additional work to show that the first surface (endowed with its canonical conformal structure) and the second surface (endowed with its canonical length space structure) agree in law. There are many variants of both LQG and TBM (which may differ in genus, boundary structure, and number and type of special marked points).  See for example the recently announced work in preparation by Bettinelli and Miermont (on the Brownian map side) \cite{bm_compact2} or the work by Guillarmou, Rhodes and Vargas (on the LQG side) in \cite{guillarmourhodesvargas}.  It remains an open problem to show that these surfaces described in these papers are equivalent. We only note here that all of these variants can now be defined as random variables on the same space --- the space of (possibly marked) {\em conformal mm-spaces} --- and that comparing the laws of surfaces defined on the LQG side with the laws of surfaces defined on the Brownian map side is therefore at least possible in principle.

\subsection{Main results from previous papers}
\label{subsec::prior_main_results}

Before we proceed, let us recall briefly the approach taken in the first two papers in this series. The first paper \cite{qlebm} considers a sequence of points $(x_n)$ chosen i.i.d.\ from the measure on an instance $\CS$ of the $\sqrt{8/3}$-LQG sphere, and shows how to define a metric~$\qdist$ on $(x_n)$. This is accomplished by letting $\qdist(x_i, x_j)$ denote the amount of time it takes for a certain growth process, started at $x_i$, to absorb the point $x_j$. This growth process is a form of the {\em quantum Loewner evolution with parameters} $(\gamma^2,\eta) = (8/3,0)$ (i.e.\ $\QLE(8/3,0)$) as introduced by the authors in \cite{ms2013qle}. We showed in \cite{qlebm} that~$\qdist$ is a.s.\ determined by $\CS$ together with the set $(x_j)$, i.e., it does not depend on any additional randomness associated with the growth processes.

The second paper in the current series \cite{qle_continuity} shows that there is a.s.\ a unique way to continuously extend the metric~$\qdist$ on $(x_j)$ to a metric $\oqdist$ defined on the entire sphere~$\CS$. It is also shown that the identity map from~$\CS$ (endowed with the Euclidean metric $\dsphere$) to itself (endowed with the metric $\oqdist$) is a.s.\ H\"older continuous in both directions.  It was further shown in \cite{qle_continuity} that as a random mm-space, the pair $(\CS,\oqdist)$ agrees in law with TBM.  Along the way to proving this result, we showed that $(\CS,\oqdist)$ is geodesic, i.e.\ it is a.s.\ the case that every pair of points $x,y$ can be connected by a path whose length with respect to $\oqdist$ is equal to $\oqdist(x,y)$. We recall below three major statements obtained in \cite{qle_continuity}.

In making these statements, as just above, we use $\dsphere$ to denote the standard Euclidean path-length metric on $\s^2$ (when $\s^2$ is embedded in $\R^3$ in the standard way).

\begin{theorem}[\cite{qle_continuity}]
\label{thm::continuity}
Suppose that $\CS = (\s^2,h)$ is a unit area $\sqrt{8/3}$-LQG sphere, $(x_n)$ is an i.i.d.\ sequence chosen from the quantum measure on $\CS$, and $\qdist$ is the associated $\QLE(8/3,0)$ metric on $(x_n)$.  Then $(x_i,x_j) \mapsto \qdist(x_i,x_j)$ is a.s.\ H\"older continuous with respect to the metric $\dsphere^2( (x_1, y_1), (x_2, y_2) ) = \dsphere(x_1, x_2) + \dsphere(y_1,y_2)$ that $\dsphere$ induces on the product space $\s^2 \times \s^2$.  In particular, $\qdist$ uniquely extends to a H\"older continuous (w.r.t.\ $\dsphere^2$) function $\oqdist$ on $\s^2 \times \s^2$.
\end{theorem}

The following states that $\oqdist$ a.s.\ induces a metric on $\s^2$ which is isometric to the metric space completion of $\qdist$.

\begin{theorem}[\cite{qle_continuity}]
\label{thm::metric_completion}
Suppose that $\CS = (\s^2,h)$ is a unit area $\sqrt{8/3}$-LQG sphere and that $\oqdist$ is as in Theorem~\ref{thm::continuity}.  Then $\oqdist$ a.s.\ defines a metric on $\s^2$ which is isometric to the metric space completion of $\qdist$.  Moreover, the identity map between $(\s^2, \dsphere)$ and $(\s^2, \oqdist)$ is a.s.\ H\"older continuous in both directions.
\end{theorem}

Throughout the rest of the paper, in order to lighten the notation, we will write $\qdist$ instead of $\oqdist$.  We will also make use of the following notation.  We will write $\ball{z}{\epsilon}$ for the open Euclidean ball centered at $z$ of radius $\epsilon$ and write $\qball{z}{\epsilon}$ for the ball with respect to $\qdist$.  When we have fixed a given reference point $w$, we will write $\qhull{z}{\epsilon}$ for the hull or filling of $\qball{z}{\epsilon}$ relative to $w$.  That is, $\qhull{z}{\epsilon}$ is the set of points disconnected from $w$ by the closure of $\qball{z}{\epsilon}$.  Typically, $w = \infty$.  For a set $K$, we also let $\qdiam(K)$ be the diameter of $K$ with respect to $\qdist$.

\begin{theorem}[\cite{qle_continuity}]
\label{thm::tbm_lqg}
Suppose that $\CS = (\s^2,h)$ is a unit area $\sqrt{8/3}$-LQG sphere and that $\qdist$ is as in Theorem~\ref{thm::continuity}.  Then the law of the random mm-space $(\s^2, \qdist,\mu_h)$ is the same as the law of TBM.  If $\CS = (\D,h)$ is instead a unit boundary length $\sqrt{8/3}$-LQG disk, then the law of $(\D,\qdist,\mu_h)$ is the same as that of a Brownian disk with unit boundary length.  Finally, if $\CS = (\C,h)$ is a $\sqrt{8/3}$-LQG cone, then the law of $(\C,\qdist,\mu_h)$ is the same as the law of TBP.
\end{theorem}

It is also deduced in \cite{gm2016uihpq} that the mm-space structure associated with a weight-$2$ quantum wedge agrees in law with that of the so-called Brownian half-plane, in which it is also deduced from \cite{bettinelli_miermont} that uniformly random quadrangulations of the upper half-plane converge to the Brownian half-plane.  (See also \cite{bmr2016classification} for more on the Brownian half-plane, as well a more general set of convergence results of this type.)  The work \cite{gm2016uihpq} is part of a series of papers which also includes \cite{gm2016saw,gm2016gluing} and identifies the scaling limit of the self-avoiding walk on random quadrangulations with $\SLE_{8/3}$ on $\sqrt{8/3}$-LQG.

\subsection{Main result of current paper}
\label{subsec::main_results}

Theorem~\ref{thm::tbm_lqg} implies that there exists a coupling of the law of a $\sqrt{8/3}$-LQG unit area quantum sphere $\CS$ and an instance $\CBM = (M,d,\mu)$ of TBM such that the mm-space $(\s^2,\qdist,\mu_h)$ associated with $\CS$ is a.s.\ isometric to $(M,d,\mu)$.  Moreover, by the construction of $\qdist$ given in \cite{qlebm} we have that $(\s^2,\qdist,\mu_h)$ and hence $(M,d,\mu)$ is a.s.\ determined by $\CS$.  That is, $\CBM$ is a measurable function of $\CS$.  Moreover, the analogous facts also hold in the setting of the Brownian disk, TBP, and the Brownian half-plane.  As mentioned above, the main result of this paper is that the converse also holds.

\begin{theorem}
\label{thm::map_determines_embedding}
Suppose that $\CS = (\s^2,h)$ is a $\sqrt{8/3}$-LQG unit area sphere, that $\qdist$ is as in Theorem~\ref{thm::continuity}, that $\CBM$ is the instance of TBM described by $(\s^2, \qdist,\mu_h)$, and that $\phi$ is the identity homeomorphism from $\CS$ to $\CBM$. Then $\CS$ and $\phi$ are a.s.\ determined by $\CBM$. In particular, this means that $\CS$ viewed as a random variable taking values in the space of quantum surfaces is a.s.\ determined by $\CBM$ viewed as a random variable taking values in the space of mm-spaces.

The analogous statement holds in the setting of the coupling of the Brownian disk and a quantum disk, TBP and a quantum cone, and the Brownian half-plane and a weight-$2$ quantum wedge.
\end{theorem}

\subsection{Proof strategy}
\label{subsec::outline}

Our proof of Theorem~\ref{thm::map_determines_embedding} is a variant of the proof given in
\cite[Section~10]{dms2014mating} for the fact that a certain pair of continuum random trees determines the conformal structure of the surface obtained by gluing those trees together.  Indeed, the reader might wish to read \cite[Section~10]{dms2014mating} before reading the current paper. The goal of this paper is essentially the same as the goal in \cite[Section~10]{dms2014mating} except that the pair of trees in question is defined in a different way --- in terms of the coordinates of the head of a Brownian snake process, instead of the coordinates of a correlated two dimensional Brownian motion (see Section~\ref{sec::preliminaries}).  In what follows, we will give an overview of the strategy to prove this result, which parallels the strategy used in \cite[Section~10]{dms2014mating}.

In both stories, we {\em first} give a procedure for constructing a quantum surface decorated by the pair of continuum trees (and a space-filling path forming an ``interface'' between the two trees) and {\em subsequently} show that in this construction, the pair of trees --- viewed simply as a coupled pair of random metric trees --- actually determine the conformal structure.  In both stories, the basic idea is to show that if one {\em conditions} on the pair of trees and resamples the conformal structure (conditioned on the pair of trees) then it is a.s.\ the case that nothing changes, i.e., the conformal structure is a.s.\ the same as before.

In both \cite[Section~10]{dms2014mating} and the current paper, one begins by fixing a value of $\gamma \in (0,2)$ and considering a particular kind of infinite volume $\gamma$-LQG surface, namely the so-called \emph{$\gamma$-quantum cone}.  We recall that in a $\gamma$-quantum cone, the origin looks like a ``typical'' point from the quantum measure (see Section~\ref{sec::preliminaries}).

In \cite[Section~10]{dms2014mating}, the value of $\gamma \in (0,2)$ is arbitrary and the curve $\eta'$ is an independent space-filling form of $\SLE_{16/\gamma^2}$ \cite{MS_IMAG4} that forms the interface between two continuum random trees; it is possible to make sense of those two trees as metric spaces, and the contour functions describing those trees are a pair of Brownian processes $(X, Y)$ that are correlated to an extent that depends on $\gamma$.

In the current setting, we fix $\gamma = \sqrt{8/3}$, and the curve is the interface between a continuum geodesic tree and its dual tree.  As explained in \cite{qlebm, qle_continuity}, one may endow an instance of the $\sqrt{8/3}$-quantum cone (with parameter given by $\gamma = \sqrt{8/3}$) with a metric $\qdist$ such that the resulting metric space has the law of TBP, as defined in \cite{cl2012brownianplane}. We recall that it is a.s.\ the case that for almost all points in TBP, there is a unique geodesic from that point to the special point ``at $\infty$'' in TBP and that for almost all pairs of points, the geodesics to $\infty$ merge in finite time, so that one has a continuum geodesic tree.  The law of this tree and its dual tree can be described directly.  First, the dual tree $\CT$ is an infinite continuum random tree (CRT) \cite{ald1991crt1,ald1991crt2,ald1993crt3} and has a contour function given by a Brownian motion $X$ indexed by $\R$. Let $\rho_\CT$ be the associated map from $\R$ to $\CT$. The geodesic tree then has a contour function given by a process $Y$, whose conditional law given $X$ is determined by the fact that $Y_{\rho_\CT^{-1}(s)}$ is a Brownian motion indexed by $s \in \CT$ (and its definition does not depend on which inverse of $\rho_{\CT}$ one uses, i.e., $Y_a = Y_b$ whenever $\rho_\CT(a) = \rho_{\CT}(b)$).  The map $\Gamma$ from a real time $t$ to the corresponding point in TBP is a space-filling curve, indexed by $t \in \R$, that fills all of TBP and traces a unit of area in a unit of time.  (We will review the definition of TBP more carefully in Section~\ref{subsec::map_plane}.)

In both \cite[Section~10]{dms2014mating} and the current paper, we parameterize this quantum cone by the whole plane $\C$ using the ``smoothed canonical embedding'' described in \cite[Section~10]{dms2014mating}. We recall that in this embedding, the origin and $\infty$ correspond to distinguished marked points of the quantum cone, the rotation is uniformly random, and the scaling is chosen in such a way that $(h, \phi) = 0$, where $\phi$ is a particular smooth rotationally invariant function from $\C$ to $\R$.

We emphasize that the joint law of the pair $(h,\Gamma)$ of the present paper is very different from the joint law of the pair $(h,\eta')$ from \cite[Section~10]{dms2014mating}.  Indeed, in the setting of \cite[Section~10]{dms2014mating} we have that $\eta'$ is a space-filling $\SLE_{16/\gamma^2}$ which is first sampled independently of $h$ and then reparameterized so that it fills one unit of quantum area in each unit of time.  In contrast, the construction of $\qdist$ implies that the curve $\Gamma$ is a.s.\ determined by $h$.  Moreover, $\Gamma$ serves to encode the metric structure associated with $h$ while $\eta'$ should be thought of as encoding the scaling limit of a statistical physics model (e.g., a UST instance) on a random planar map.

Let $\CU = \Gamma([0,1])$ denote the unit area portion of TBP traced by $\Gamma$ during the time interval $[0,1]$.  In light of the embedding described above, we may also view $\CU$ as a subset of $\C$. Now fix some large $n \in \N$ and for $k \in \{1,\ldots,n\}$, let $\CU_k = \Gamma([(k-1)/n,k/n])$ denote the portion traced during the interval $[(k-1)/n, k/n]$.

\begin{lemma}
\label{lem::brownianplaneradiusbound}
There exists a constant $c > 0$ such that the following is true.  For each $n \in \N$ and $k \in \{1,\ldots,n\}$, the probability that the diameter of $\CU_k$ with respect to the metric of TBP exceeds $a n^{-1/4}$ decays at least as fast as $e^{- c a^{4/3}}$. 
\end{lemma}

Lemma~\ref{lem::brownianplaneradiusbound} (stated as Proposition~\ref{prop::path_chunk_diameter_bound} and proved below) is a corollary of a large deviations principle for the Brownian snake established in \cite{serlet_ldp} which states that if~$d^*$ is the diameter of TBM with unit area then $-\log \p[d^* \geq r]$ behaves like a constant times~$r^{4/3}$ as $r \to \infty$.  Lemma~\ref{lem::brownianplaneradiusbound} also implies that for each $\epsilon > 0$ fixed, the probability that even one of the $\CU_k$ has diameter larger than $n^{-1/4 + \epsilon}$ tends to zero superpolynomially as a function of $n$.  For each $k,n$, let $\wt \CU_k$ denote the filled metric ball of radius $n^{-1/4 + \epsilon}$ centered at $\Gamma(k/n)$.

\begin{lemma}
\label{lem::momentbound}
There exists constants $c > 0$, $\alpha > 4$, and a positive probability event $\CA$ which depends only on the quantum surface parameterized $\C \setminus \CU$ such that the following is true.  For each $n \in \N$ and $k \in \{1,\ldots,n\}$, we have that
\[ \E[ \diam(\wt{\CU}_k)^4 \one_{\CA} ] \leq c n^{-\alpha}.\]
\end{lemma}

Lemma~\ref{lem::momentbound} is stated carefully as Lemma~\ref{lem::metric_ball_stable_moment_bound} and proved below.

Suppose that $K \subseteq \C$ is a compact set.  Recall that~$K$ is said to be \emph{conformally removable} if the following is true.  Suppose that $U,V \subseteq \C$ are open subsets with $K \subseteq U$.  If $\varphi \colon U \to V$ is a homeomorphism which is conformal on $U \setminus K$ then~$\varphi$ is conformal on~$U$.  The notion of conformal removability is important in the context of LQG, see e.g.\ \cite{SHE_WELD}.

\begin{lemma}
\label{lem::removableunion}
For each $n \in \N$, the set $\cup_{k=1}^n \partial \wt \CU_k$ is a.s.\ conformally removable.
\end{lemma}

Lemma~\ref{lem::removableunion} follows from Theorem~\ref{thm::finite_unions_removable} which is stated and proved in Appendix~\ref{app::removability}.  We recall that a simply connected domain $U \subseteq \C$ is said to be a \emph{H\"older domain} if there exists a conformal map $\varphi \colon \D \to U$ which is H\"older continuous up to $\partial \D$.  Establishing Lemma~\ref{lem::removableunion} will require us to rewrite, in a slightly more general way, some of the arguments for the conformal removability of H\"older domain boundaries that appear in \cite{js2000remove}.  (See Appendix~\ref{app::removability}.)

\begin{lemma}
\label{lem::insideoutsideindependent}
Given the mm-space structure of TBP, the conformal structure associated with the inside and the outside of a filled metric ball of a fixed radius centered at the origin are conditionally independent.
\end{lemma}

Lemma~\ref{lem::insideoutsideindependent} is stated carefully as Lemma~\ref{lem::inside_outside_independent} below.

Once we have the four lemmas above, Theorem~\ref{thm::map_determines_embedding} will follow from the argument used to prove the rigidity of the conformal structure given the peanosphere structure in \cite[Section~10]{dms2014mating}.

\subsection{Outline}

We will review basic definitions in Section~\ref{sec::preliminaries} and provide the proof of Theorem~\ref{thm::map_determines_embedding} in Section~\ref{sec::map_determines_field}.  Like the other papers in this series, this paper builds on several previous works by the current authors about LQG surfaces \cite{SHE_WELD, dms2014mating, quantum_spheres}, TBM \cite{map_making}, and imaginary geometry \cite{MS_IMAG,MS_IMAG2,MS_IMAG3,MS_IMAG4}. These in turn rely on sizable literatures on SLE, TBM, LQG and other topics, which are briefly reviewed in the first paper of the current series \cite{qlebm}.

\section{Preliminaries}
\label{sec::preliminaries}

In this section we will review a few preliminaries, including background on mm-spaces (Section~\ref{subsec::mm_spaces}), quantum surfaces (Section~\ref{subsec::quantum_surfaces}), a brief review of the Brownian map and plane (Section~\ref{subsec::map_plane}), and finally a few basic estimates that will be used in the proofs of our main theorems (Section~\ref{subsec::external_estimates}).

\subsection{Metric measure spaces}
\label{subsec::mm_spaces}

A \emph{metric measure space} (mm-space) is a metric space $(M,d)$ together with a measure $\mu$ on the Borel $\sigma$-algebra associated with $(M,d)$.  In this article, we will be considering a number of \emph{random} mm-spaces and we will also want to consider conditional probabilities where the $\sigma$-algebra in question is generated by such a random mm-space.  In order to make this precise, we need to specify a $\sigma$-algebra on mm-spaces.  There are several possibilities that one could choose from here.  In order to be consistent with \cite{map_making}, we will use the Borel $\sigma$-algebra generated by the so-called Gromov-weak topology.  (See \cite[Section~2.4]{map_making} and the references therein for additional detail.)

Suppose that $(M,d,\mu)$ is a mm-space with $\mu$ a probability measure and let $\E_\mu$ denote the expectation associated with the matrix of distances $d_{ij} = d(x_i,x_j)$ where $(x_j)$ is an i.i.d.\ sequence sampled from $\mu$.  The Gromov-weak topology is the weakest topology such that for each $k \in \N$ and each bounded, continuous function $\psi$ on $\R^{k^2}$ we have that the map
\[ (M,d,\mu) \mapsto \E_\mu[ \psi(M_k)]\]
is measurable where $M_k = (d_{ij})_{i,j=1}^k$.

More generally, the Gromov-weak topology on mm-spaces $(M,d,\mu,z_1,\ldots,z_n)$ with $n$ marked points is defined in the same way except we take the sequence $(x_j)$ so that $x_1=z_1,\ldots,x_n=z_n$ and the remaining elements of the sequence to be i.i.d.\ from $\mu$.  We refer the reader to \cite{grevenpfaffelhuberwinter} for additional background on the Gromov-weak topology.

We let $\CM$ be the Borel $\sigma$-algebra generated by the Gromov-weak topology.  More generally, we let $\CM^n$ be the Borel $\sigma$-algebra generated by the Gromov-weak topology with $n$ marked points.

We note that, using this topology, two mm-spaces $(M_i,d_i,\mu_i)$ for $i=1,2$ are equivalent if there exists a measure preserving map $\psi \colon M_1 \to M_2$ which restricts to an isometry on the support of $\mu_1$.

In \cite[Section~2.4]{map_making}, a number of properties of the Gromov-weak topology are recorded.  For example, it is shown there that the following events are Gromov-weak measurable:
\begin{itemize}
\item $(M,d)$ is compact (more precisely this means that $(M,d,\nu)$ is equivalent to some $(\wt M, \wt d, \wt \nu)$ for which $(\wt M, \wt d)$ is compact)  \cite[Proposition~2.12]{map_making},
\item $(M,d)$ is compact and geodesic \cite[Proposition~2.14]{map_making}, and 
\item $(M,d)$ is geodesic and homeomorphic to $\s^2$ \cite[Proposition~2.15]{map_making}.
\end{itemize}

We note that TBM with unit area can be viewed as a random variable taking values in the space of mm-spaces using the $\sigma$-algebra $\CM$ as described just above.

We finish this subsection with two observations:
\begin{itemize}
\item If $(M,d,\mu)$ is an mm-space where $0 < \mu(M) < \infty$, then we can let $\mu^* = \mu/\mu(M)$ and encode $(M,d,\mu)$ by the pair consisting of $(M,d,\mu^*)$ and $\mu(M)$.  Therefore such a space fits into the aforementioned framework, where we add an extra variable to keep track of the total mass $\mu(M)$ and the natural $\sigma$-algebra in this context is the product $\sigma$-algebra associated with $\CM$ and the Borel $\sigma$-algebra on~$\R$.  We note that TBM with random area can be viewed as a random variable taking values in the product of the space of mm-spaces as described above and~$\R$ with this $\sigma$-algebra.
\item Suppose that $(M,d,\mu,z)$ is a marked mm-space where $\mu(M) = \infty$ but $0 < \mu(B(z,r)) < \infty$ for all $r > 0$ and $M = \cup_{r > 0} B(z,r)$.  For each $r > 0$, let $\mu_r^*$ be given by $\mu|_{B(z,r)} / \mu(B(z,r))$.  For $r > 0$, each of the spaces $(B(z,r),d,\mu_r^*,z)$ together with $\mu(B(z,r))$ fits into the framework described just above and we can view such a space $(M,d,\mu,z)$ as a sequence $(B(z,n),d,\mu_n^*,z)$ together with $\mu(B(z,n))$ each of which fits into the framework from just above.  This leads to a natural $\sigma$-algebra in this setting and this is the space in which we will view TBP as taking values.
\end{itemize}

\subsection{Quantum surfaces}
\label{subsec::quantum_surfaces}

Suppose that $h$ is an instance of some form of the Gaussian free field (GFF) on a planar domain $D$ and fix $\gamma \in (0,2)$.  The $\gamma$-Liouville quantum gravity (LQG) measure associated with $h$ is formally given by $\mu_h = e^{\gamma h} dz$ where $dz$ denotes Lebesgue measure on $D$.  Since $h$ is a distribution and not a function, this expression requires interpretation and can be made rigorous using a regularization procedure.  Similarly, if $L$ is a linear segment of the boundary, then the $\gamma$-LQG boundary measure is given by $\nu_h = e^{\gamma h/2} dz$ where here $dz$ denotes Lebesgue measure on $L$.  This expression is made rigorous using the same regularization procedure used to construct $\mu_h$.

Suppose that $\wt{D}$ is another planar domain and $\varphi \colon \wt{D} \to D$ is a conformal transformation.  If one defines
\begin{equation}
\label{eqn::quantum_equivalence}
 \wt{h} = h \circ \varphi + Q \log|\varphi'| \quad\text{where}\quad Q = \frac{2}{\gamma} + \frac{\gamma}{2}	
\end{equation}
then $\mu_{\wt{h}}(A) = \mu_h(\varphi(A))$ (resp.\ $\nu_{\wt{h}}(A) = \nu_h(\varphi(A))$) for all $A \subseteq \wt{D}$ (resp.\ $A \subseteq \partial \wt{D}$) Borel.  This in particular allows one to make sense of $\nu_h$ on segments of $\partial D$ which are not necessarily linear.  A quantum surface is an equivalence class where a representative consists of a pair $(D,h)$ and two pairs $(D,h)$, $(\wt{D},\wt{h})$ are said to be equivalent if~$h$ and~$\wt{h}$ are related as in~\eqref{eqn::quantum_equivalence}.  We refer to a representative $(D,h)$ of a quantum surface as an \emph{embedding}.

More a generally, a \emph{marked quantum surface} is an equivalence class where a representative consists of a triple $(D,h,\ul{z})$ where $\ul{z} = (z_1,\ldots,z_k)$ is a collection of points in $\ol{D}$ and two marked quantum surfaces $(D,h,\ul{z})$, $(\wt{D},\wt{h},\ul{\wt{z}})$ are said to be equivalent if $h$, $\wt{h}$ are related as in~\eqref{eqn::quantum_equivalence} and $\varphi(\wt{z}_j) = z_j$ for all $1 \leq j \leq k$.

In this work, we shall be primarily interested in several types of quantum surfaces: quantum wedges, cones, disks, and spheres.  As the definitions of these surfaces were reviewed in earlier works in this series, we will not repeat them here and instead refer the reader to \cite{dms2014mating}.

\subsection{The Brownian map and plane}
\label{subsec::map_plane}

We now recall the definition of the Brownian map (TBM) \cite{lg2013uniqueness,mier2013bm} and the Brownian plane (TBP) \cite{cl2012brownianplane}.  We first begin by reminding the reader of the construction of the so-called Brownian snake process $Y$ on $[0,1]$.  One can sample from the law of the Brownian snake using the following two-step procedure:
\begin{enumerate}
\item Sample a Brownian excursion $X \colon [0,1] \to \R_+$ with $X(0) = X(1) = 0$ and let
\begin{equation}
\label{eqn::m_e_def}
	 m_X(s,t) = \inf_{u \in [s,t]} X(u) \quad\text{for}\quad s < t.
\end{equation}
\item Given $X$, sample a mean-zero Gaussian process $Y$ with covariance given by $\cov{Y_s}{Y_t} = m_X(s,t)$ for $s < t$.
\end{enumerate}

TBM (with unit area) is the random mm-space which is constructed from the Brownian snake as follows.  For $s,t \in [0,1]$, we set
\[ d^\circ(s,t) = Y_s + Y_t - 2\max\left( \inf_{u \in [s,t]} Y_u,\inf_{u \in [t,s]} Y_u \right).\]
Here, we assume without loss of generality that $s < t$ and take $[t,s] = [0,s] \cup [t,1]$.  Let $\CT$ be the instance of the continuum random tree (CRT) associated with $X$ and let $\rho_\CT \colon [0,1] \to \CT$ be the projection map.  For $a,b \in \CT$, we then set
\begin{equation}
\label{eqn::d_circ_def}
	d_\CT^\circ(a,b) = \inf\{ d^\circ(s,t) : \rho_\CT(s) = a, \rho_\CT(t) = b\}.
\end{equation}
Finally, for $a,b \in \CT$, we set
\begin{equation}
\label{eqn::d_metric_def}
d(a,b) = \inf\left\{ \sum_{j=1}^k d_\CT^\circ(a_{j-1},a_j) \right\}
\end{equation}
where the infimum is taken over all $k \in \N$ and $a_0 = a,a_1,\ldots,a_k=b$ in $\CT$.  The unit-area Brownian map is the mm-space $(M,d,\mu)$ where $(M,d)$ is given by the metric quotient $\CT / \cong$ where $a \cong b$ if and only if $d(a,b) = 0$ and $\mu$ is the pushforward of Lebesgue measure from $[0,1]$ to $(M,d)$ under the projection map $\rho \colon [0,1] \to M$.  This construction naturally associates with $(M,d,\mu)$ a space-filling, non-crossing path $\Gamma$ where $\Gamma(t) = \rho(t)$ for $t \in [0,1]$.  (In fact, $\Gamma$ is the peano curve which snakes between a tree of geodesics and its dual tree on $(M,d)$.)  In particular, $(M,d,\mu)$ is naturally marked by $\Gamma(0)$.

Recalling that $Y_s$ represents the distance between $\rho(s)$ and the root of the geodesic tree in TBM, i.e., $\rho(s^*)$ where $s^*$ is the unique value in $[0,1]$ such that $Y_{s^*} = \inf_{s \in [0,s]} Y_s$, it follows that the diameter $d^*$ of $(M,d)$ satisfies the bounds
\[ \sup_{s \in [0,1]} Y_s - \inf_{s \in [0,1]} Y_s \leq d^* \leq 2\left(\sup_{s \in [0,1]} Y_s - \inf_{s \in [0,1]} Y_s\right).\]
The precise asymptotics for the tail of $d^*$ were determined by Serlet \cite[Proposition~14]{serlet_ldp}.  We restate a variant of this result here because it will be important for our later arguments.  (The result in \cite{serlet_ldp} in fact includes matching upper and lower bounds.)
\begin{proposition}[\cite{serlet_ldp}]
\label{prop::serlet}
There exists a constant $c_0 > 0$ such that, as $r \to \infty$, we have that
\begin{equation}
\label{eqn::serlet_ldp}
	\log \p[ d^* \geq r] \leq -c_0 r^{4/3}.
\end{equation}
\end{proposition}

We observe that if we have a marked instance of TBM $(M,d,\mu,\Gamma(0))$, then for each $t \in [0,1]$ fixed we have that $(M,d,\mu,\Gamma(0)) \stackrel{d}{=} (M,d,\mu,\Gamma(t))$.  That is, the law of TBM is invariant under re-rooting according to $\Gamma$ (note that $\Gamma(0)$ is sometimes called the dual root of TBM).  In fact, the conditional law of $\Gamma(0)$ given $(M,d,\mu)$ is equal to $\mu$.

In this work, we will be interested in TBP \cite{cl2012brownianplane}.  We will now remind the reader of the construction of TBP and explain why Serlet's bound can be used to deduce diameter bounds for the corresponding space-filling path.

Following \cite{cl2012brownianplane}, one can produce a sample from the law of TBP as follows:
\begin{enumerate}
\item Sample the process $X \colon \R \to \R_+$ by taking $X_t = R_t$ for $t \geq 0$ (resp.\ $X_t = R_{-t}'$ for $t \leq 0$) where $R,R'$ are independent $3$-dimensional Bessel processes starting from~$0$.
\item Given~$X$, let~$Y$ be the mean-zero Gaussian process with covariance $\cov{Y_s}{Y_t} = m_X(s,t)$ where $m_X$ is as in~\eqref{eqn::m_e_def} (with $X$ as in the previous item).
\end{enumerate}
We note that the two $3$-dimensional Bessel processes in the definition of $X$ correspond to zooming in near the start and end of the Brownian excursion in the definition of TBM.

The mm-space structure of TBP is constructed in a manner which is analogous to TBM.  Also associated with it is a space-filling path $\Gamma$ which is given by projecting the identity map on $\R$.  Moreover, the law of TBP is invariant under re-rooting according to $\Gamma$, as it is the infinite volume limit of TBM which enjoys the corresponding property.  See, e.g., \cite[Proposition~4]{cl2012brownianplane}.

We are now going to explain why Serlet's tail bound (Proposition~\ref{prop::serlet}) extends to bound the diameter $d_{s,t}^*$ of the sets $\Gamma([s,t])$ for $s < t$ of the space-filling path on TBP.

\begin{proposition}
\label{prop::path_chunk_diameter_bound}
There exists a constant $c_0 > 0$ such that, as $r \to \infty$, we have that
\begin{equation}
\label{eqn::space_filling_chunk_size}
	\log \p[ d_{0,1}^* \geq r] \leq -c_0 r^{4/3}.
\end{equation}
More generally, for each $s < t$ we have that
\begin{equation}
\label{eqn::space_filling_chunk_size_general}
	\log \p[ d_{s,t}^* \geq r] \leq -c_0 (t-s)^{-1/3} r^{4/3}.
\end{equation}
\end{proposition}
\begin{proof}
We note that~\eqref{eqn::space_filling_chunk_size_general} follows from~\eqref{eqn::space_filling_chunk_size} and the scaling properties of TBP.  In particular, we have that $d_{s,t}^* \stackrel{d}{=} (t-s)^{1/4} d_{0,1}^*$ for all $s < t$.  This leaves us to prove~\eqref{eqn::space_filling_chunk_size}.

We will deduce~\eqref{eqn::space_filling_chunk_size} from Proposition~\ref{prop::serlet}.  By the definition~\eqref{eqn::d_circ_def}, \eqref{eqn::d_metric_def} of the metric for TBP and the definition of the space-filling path $\Gamma$, we know that
\begin{equation}
\label{eqn::path_diam_bound}
d_{0,1}^* \leq 2 \left(\sup_{s \in [0,1]} Y_s - \inf_{s \in [0,1]} Y_s \right).
\end{equation}
Therefore to extract~\eqref{eqn::space_filling_chunk_size} from~\eqref{eqn::serlet_ldp}, we just need to control the Radon-Nikodym derivative of the law of a $3$-dimensional Bessel process run for time $1/2$ with respect to the law of a Brownian excursion run for time $1/2$.  Indeed, this will imply the analog of~\eqref{eqn::path_diam_bound} with $[0,1/2]$ in place of $[0,1]$.  The case with $[0,1]$ then follows from the case with $[0,1/2]$ by applying scaling.

We begin by observing that if $R$ is a $3$-dimensional Bessel process starting from $0$ and $X$ is a Brownian excursion from $0$ to $0$, then for each $x \geq 0$ we have that the conditional law of $R|_{[0,1/2]}$ given $R_{1/2} = x$ is equal to the conditional law of $X|_{[0,1/2]}$ given $X_{1/2} = x$.  Therefore we just have to control the Radon-Nikodym derivative of the law of $R_{1/2}$ with respect to the law of $X_{1/2}$.

By \cite[Chapter~XII, Theorem~4.1]{RY04} and \cite[Chapter~XI]{RY04} the density of the laws of $X_{1/2}$ and $R_{1/2}$, respectively, with respect to Lebesgue measure on $\R_+$ are given by the following expressions evaluated at $t=1/2$:
\begin{equation}
\label{eqn::densities}
\frac{2 x^2}{(2\pi t^3(1-t)^3)^{1/2}} \exp\left( - \frac{x^2}{2t} - \frac{x^2}{2(1-t)} \right) \quad\text{and}\quad \frac{2 x^2}{(2 \pi t^3)^{1/2}} \exp\left(- \frac{x^2}{2t} \right).
\end{equation}
Therefore the claimed bounds follow by using the explicit form of the Radon-Nikodym derivative and applying H\"older's inequality.  More precisely, if $r > 0$ and $A_r$ is the event that the right hand side of~\eqref{eqn::path_diam_bound} is at least $r$, with $[0,1/2]$ in place of $[0,1]$, and $\mu$ (resp.\ $\nu$) denotes the law of TBM (resp.\ TBP), then for conjugate exponents $p,q \geq 1$ we have that
\[ \nu[A_r] = \int Z \one_{A_r} d \mu \leq \left( \int  Z^p d\mu \right)^{1/p} \mu[A_r]^{1/q}\]
where $Z$ is given by the second expression in~\eqref{eqn::densities} divided by the first with $t=1/2$.  Proposition~\ref{prop::serlet} gives us the rate of decay for $\mu[A_r]$ as $r \to \infty$; note that taking the $1/q^{\rm th}$ power only changes the constant in  front of the term $r^{4/3}$.  The result thus follows because by inspecting the explicit form of $Z$, we see that we can choose $p > 1$ sufficiently close to $1$ so that $\int Z^p d\mu < \infty$.
\end{proof}

The boundary of the range $\Gamma([s,t])$ of the space-filling path $\Gamma$ on the interval $[s,t]$ can be divided into four boundary segments: two branches of the geodesic tree (i.e., geodesics) and two branches of the dual tree.  These sets are not known to be conformally removable, after embedding into $\C$ using the $\QLE(8/3,0)$ metric ({\it Update:} The conformal removability of LQG geodesics was recently proved in \cite{mq2020geodesics}, after this work appeared.)  For this reason, in our proof of Theorem~\ref{thm::map_determines_embedding} we will cover each such chunk by a filled metric ball and instead use the removability of filled metric ball boundaries \cite{ms2013qle}.  (This is in contrast to \cite[Section~10]{dms2014mating}, in which it was known that the boundaries of the chunks of the space-filling path were conformally removable.)

\subsection{Variance and distortion estimates}
\label{subsec::external_estimates}

We now record a variance bound and a distortion estimate which will be important for the proof of Theorem~\ref{thm::map_determines_embedding}.  The first estimate that we will state is a general bound, the so-called Efron-Stein inequality \cite{efronstein}, on the variance of a function of independent random variables.

\begin{lemma}
\label{lem::variance_bound}
Let $A = A(X_1,\ldots,X_n)$ be a function of $n$ independent random variables $X_1,\ldots,X_n$ such that $\ex{ A^2 } < \infty$.  Then
\begin{equation}
\label{eqn::variance_bound}	
\var{A} \leq \sum_{i=1}^n \ex{ \var{ A \giv \CF_i}}  \quad\text{where}\quad \CF_i = \sigma(X_j : j \neq i).
\end{equation}
\end{lemma}

This result is stated in \cite[Lemma~10.4]{dms2014mating} together with an independent proof.

When we apply Lemma~\ref{lem::variance_bound} below, the role of the $X_i$'s will be played by quantum surfaces parameterized by hulls of metric balls and $A$ will be a function which determines the embedding of the quantum surface formed by the $X_i$'s and the surface formed by the complement of their union.  We will compute the variance conditionally on the latter, so that $A$ is just a function of the $X_i$'s.

The other estimate that we will state is a general estimate for conformal maps and is stated in \cite[Lemma~10.5]{dms2014mating}.  Before we recall this estimate, we first remind the reader of the following definition.  Suppose that $K \subseteq \C$ is a compact hull and let $F \colon \C \setminus \ol{\D} \to \C \setminus K$ be the unique conformal map with $F(\infty) = \infty$ and $F'(\infty) > 0$.  Then we can write
\[ F(z) = a_{-1} z + a_0 + \sum_{n=1}^\infty a_n z^{-n}.\]
We will refer to $a_0$ as the \emph{harmonic center} of $K$.

\begin{lemma}
\label{lem::conformal_map_bound}
There exist constants $c_1,c_2 > 0$ such that the following is true.  Let $K_1$ be a hull of diameter at most $r$ and $K_2$ another hull such that there exists a conformal map $F \colon \C \setminus K_1 \to \C \setminus K_2$ of the form
\[ F(z) = z + \sum_{n=1}^\infty \alpha_n z^{-n}.\]
Then whenever $\dist(z,K_1) \geq c_1r$ we have that
\[ |F(z) - z| \leq c_2 r^2|z-b_1|^{-1}\]
where $b_1$ is the harmonic center of $K_1$.
\end{lemma}

In the proof of Theorem~\ref{thm::map_determines_embedding}, Lemma~\ref{lem::conformal_map_bound} will be used to bound the summands in~\eqref{eqn::variance_bound} in terms of the fourth moment of the Euclidean diameters of a collection of quantum metric balls.

\section{Brownian surfaces determine their embedding}
\label{sec::map_determines_field}

The purpose of this section is to prove Theorem~\ref{thm::map_determines_embedding}.  We begin in Section~\ref{subsec::quantum_cone_stable} by describing a type of embedding of a $\sqrt{8/3}$-quantum cone that will be useful when we complete the proof of Theorem~\ref{thm::map_determines_embedding} in Section~\ref{subsec::resampling_argument}.

\subsection{Smooth embeddings and stability}
\label{subsec::quantum_cone_stable}

Suppose that $(\C,h,0,\infty)$ is a $\sqrt{8/3}$-quantum cone and let $\mu_h$ be the associated LQG area measure.  For each $z \in \C$ and $r >0$, let $h_r(z)$ be the average of $h$ on $\partial B(z,r)$.  In many places, it is convenient to take the embedding of the quantum cone so that the supremum of $r > 0$ values at which the process $r \mapsto h_r(0) + Q \log r$, $\gamma=\sqrt{8/3}$ and $Q$ as in~\eqref{eqn::quantum_equivalence}, takes on the value $0$ is equal to $r=1$.  This particular choice of embedding in some places is referred to as the circle-average embedding of $h$.  It is useful because it is possible to give an explicit description of the law of $h$.  When we complete the proof of Theorem~\ref{thm::map_determines_embedding} in Section~\ref{subsec::resampling_argument} below, however, we will need to consider a slightly different embedding of the quantum cone into $\C$.  The reason for this is that we will be cutting out various parts of the surface and then gluing in new pieces and the aforementioned normalization is not compatible with this operation as the circles centered at the origin will typically not be invariant under the operation of cutting/gluing.

We are now going to describe an alternative choice of embedding of a quantum cone into $\C$ that we refer to as a \emph{smooth canonical description}.  This type of embedding was introduced in \cite[Section~10]{dms2014mating} (and also made use of in \cite{ghms2015correlation}), but for completeness we will recall the definition now.  For $t \geq 0$, let
\[ R_t = \frac{1}{\gamma} \log \E[ \mu_h(\ball{0}{e^{-t}}) \giv h_{e^{-t}}(0) ].\]
That is, $R_t$ is chosen so that $e^{\gamma R_t}$ is equal to the expected $\gamma$-LQG mass in $\ball{0}{e^{-t}}$ given $h_{e^{-t}}(0)$.  Then $R_t = B_t + (\gamma-Q) t$ where $B_t$ for $t > 0$ is a standard Brownian motion \cite[Section~4]{DS08}; note that $\gamma-Q < 0$.  Fix $\epsilon > 0$ very small.  Now we will consider the quantity
\begin{equation}
\label{eqn::smooth_embedding_function}
\int_{-\epsilon}^0 R_t \phi(t) dt
\end{equation}
where $\phi \geq 0$ is a $C^\infty$ bump function supported on $(-\epsilon,0]$ with total integral one.  We apply a coordinate change rescaling to~$h$ so that the function from~\eqref{eqn::smooth_embedding_function} achieves the value~$0$ for the first time at $0$.  We refer to the $h$ with this type of scaling as the \emph{smooth canonical description} for the quantum cone.

When we perform the cutting/gluing operations in Section~\ref{subsec::resampling_argument} below, it will be important for us to work on a certain event on which the behavior of certain quantities is uniform.  Moreover, it will be important that this event is determined by the quantum surface in question and is invariant under performing certain cutting/gluing operations.  We will now make the definition of this event precise.  Fix $s > 0$ and $r > 1$.  Suppose that $(\C,h,0,\infty)$ is a $\sqrt{8/3}$-quantum cone and let $\Gamma$ be the space-filling path associated with TBP structure of $h$.  We say that $h$ is \emph{$(r,s)$-stable} if the following are true.
\begin{itemize}
\item $\Gamma([-r^{-1},r^{-1}]) \subseteq \qhull{0}{s} \subseteq \qhull{0}{2s} \subseteq B(0,\tfrac{1}{2})$
\item For every $z \in \qhull{0}{s}$ it is the case that if we translate by $-z$, the scaling factor which puts the resulting field into a smooth canonical description is between $r^{-1}$ and $r$.
\item The previous items continue to hold if we cut out the surface parameterized by $\qhull{0}{2s}$ and then weld in any surface which preserves the overall mm-space structure.  In other words, if we consider any conformal map $\psi$ from $\C \setminus \qhull{0}{2s}$ to $\C \setminus K$ for some hull $K \subseteq \C$ that is normalized so that
	\begin{enumerate}[(i)]
	\item $\psi$ fixes $\infty$,
	\item $\psi$ has positive real derivative at $\infty$ (i.e., $\lim_{z \to \infty} \psi(z)/z > 0$),
	\item $\psi$ is scaled in such a way that the pushforward of $h$ via the quantum coordinate change described by $\psi$ corresponds to a smooth canonical description
	\end{enumerate}
then $K \subseteq B(0,\tfrac{1}{2})$.  Moreover, for any $z \in K$, the scaling factor necessary to put the field into a smooth canonical description is between $r^{-1}$ and $r$.
\end{itemize}

We emphasize that $\CA_{r,s}$ is measurable with respect to the quantum surface parameterized by $\C \setminus \qhull{0}{2s}$ and the overall mm-space structure.  This property will be important for us later.

A notion of stability is also introduced in \cite[Section~10.4.2]{dms2014mating}.  The definition above differs from that in \cite{dms2014mating} in that the above is defined in terms of the mm-space structure of the surface while the version in \cite{dms2014mating} is focused on the behavior of the space-filling $\SLE$.

We are now going to show that by adjusting the parameters $r$ and $s$, the probability of the event that a $\sqrt{8/3}$-quantum cone is $(r,s)$-stable can be made arbitrarily close to $1$.

\begin{proposition}
\label{prop::stable_positive_probability}
Let $(\C,h,0,\infty)$ be a smooth canonical description of a $\sqrt{8/3}$-quantum cone.  Fix $s > 0$, $r > 1$, and let $\CA_{r,s}$ be the event that $h$ is $(r,s)$-stable.  Then $\pr{\CA_{r,s}} > 0$.  Moreover, for every $\epsilon > 0$ there exists $s_0 > 0$ such that for all $s \in (0,s_0)$ there exists $r_0 > 1$ such that $r \geq r_0$ implies that $\pr{\CA_{r,s}} \geq 1-\epsilon$.
\end{proposition}
\begin{proof}
The proof is similar to that of \cite[Proposition~10.16]{dms2014mating}, but we will include it here for completeness (though we will cite some of the results from \cite[Section~10.4.2]{dms2014mating} that we make use of verbatim).

We let $\Phi$ be the set of distorted bump functions of the form $|g'|^2 \phi \circ g$ for which $g^{-1}$ is conformal outside of $K$ for some $K \subseteq \ball{0}{\tfrac{1}{10}}$ and looks like the identity near $\infty$.  Then we know that $\Phi$ is a sequentially compact subset of the space of test functions with respect to the topology of uniform convergence of all derivatives on compact sets.  Indeed, this follows from the Arzel\'a-Ascoli theorem.  Since this space is a Fr\'echet space (thus metrizable) $\Phi$ is also compact.  In particular, for any distribution $h$, both of the quantities
\[ M_1(h) = \inf_{\wt{\phi} \in \Phi} (h, \phi - \wt{\phi}) \quad\text{and}\quad M_2(h) = \sup_{\wt{\phi} \in \Phi} (h,\phi-\wt{\phi})\]
are finite.  By continuity, they are equal to the infimum and supremum taken over a countable dense subset of $\wt{\phi} \in \Phi$.  In the case that $h$ is given by a GFF, the value of $(h,\phi-\wt{\phi})$ is a continuous function of $\wt{\phi}$ hence also has a maximum over $\Phi$.

Let $R_t^j = M_j(z \mapsto h(e^t z) + Qt)$.  By \cite[Proposition~10.19]{dms2014mating}, both of the processes $R_t^1$ and $R_t^2$ tend to $\infty$ as $t \to \infty$ and to $-\infty$ as $t \to -\infty$.  This implies that there a.s.\ exists $c \in (0,\infty)$ (random) such that we have $R_t^1, R_t^2 > 0$ for $t > c$ and $R_t^1, R_t^2 < 0$ for $t < -c$.  Pick $c_1 \in (0,\infty)$ such that with
\begin{equation}
\label{eqn::p_c1_c}
E_1 = \{c \leq c_1\} \quad\text{we have}\quad \pr{E_1} \geq 1-\epsilon.
\end{equation}
If $K \subseteq \ball{0}{\tfrac{1}{2} e^{-c_1}}$ then we have that $a K \in \ball{0}{\tfrac{1}{2}}$ for any $a \in [0,e^{c_1}]$.

We know that $\qhull{0}{2s} \subseteq \ball{0}{\tfrac{1}{2} e^{-c_1}}$ with positive probability for any value of $s > 0$.  In fact, we have that for any $\epsilon > 0$ there exists $s_0 >0$ such that with
\begin{equation}
\label{eqn::p_c_c}
E_2 = \{ \qhull{0}{s_0} \subseteq \ball{0}{\tfrac{1}{2} e^{-c_1}} \} \quad\text{we have}\quad \pr{ E_2} \geq 1-\epsilon.
\end{equation}
On $E_1,E_2$, it is not hard to check that if one swaps out $\qhull{0}{s_0}$ with any other quantum surface, then the resulting appropriately scaled surface will still belong to $\ball{0}{\tfrac{1}{2}}$.

For any fixed value of $s \in (0,s_0)$, follows from Proposition~\ref{prop::path_chunk_diameter_bound} that there exists $r_0 > 0$ such that $r \geq r_0$ implies that with
\begin{equation}
\label{eqn::p_k_r}
E_3 = \{ \qdiam(\Gamma([-r^{-1},r^{-1}]))  \leq  s\} \quad\text{we have}\quad \pr{E_3} \geq 1-\epsilon.
\end{equation}

On $E_1 \cap E_2 \cap E_3$, all of the properties for the configuration to be $(r,s)$-stable hold except for possibly the part of the event mentioned in the second bullet point.  One can verify that this condition also holds with probability at least $1-\epsilon$ using the same argument which was used to establish~\eqref{eqn::p_c1_c}.
\end{proof}

\begin{lemma}
\label{lem::metric_ball_stable_moment_bound}
There exists a constant $c_0 > 0$ such that the following is true.  Let $(\C,h,0,\infty)$ be a smooth canonical description of a $\sqrt{8/3}$-quantum cone.  Fix $s \in (0,1)$, $r > 1$, and let $\CA_{r,s}$ be the event that $h$ is $(r,s)$-stable.  Fix $t \in [-r^{-1},r^{-1}]$ and let $y = \Gamma(t)$ where $\Gamma$ is the space-filling path associated with the Brownian plane structure of $h$.  Let $a_1 > 4$ be the constant as in \cite[Equation~(4.3), Proposition~4.2]{qle_continuity}.  Then we have that
\[ \ex{ \diam(\qhull{y}{\epsilon})^4  \one_{\CA_{r,s}}} \leq c_0 r^4 \epsilon^{a_1} \quad\text{for all}\quad \epsilon \in (0,1).\]
\end{lemma}
\begin{proof}
Let $X$ be the random scaling factor necessary to put $h$ translated by $-\Gamma(t)$ into a smooth canonical description.  On $\CA_{r,s}$, we have that $X \in [r^{-1},r]$.  \cite[Proposition~4.2]{qle_continuity} is stated when $(\C,h,0,\infty)$ has the circle average embedding.  When $h$ is instead a smooth canonical description with normalizing bump function supported in $\ball{0}{2} \setminus \ball{0}{1/2}$, it is easy to see that $h_{1/2}(0)$ is a normal random variable with mean zero and bounded variance.  In particular, the probability that it is of order $c \sqrt{\log \epsilon^{-1}}$ decays like a power of $\epsilon$ that can be made arbitrarily large by making $c$ large.  Thus it is not difficult to see that the result of \cite[Proposition~4.2]{qle_continuity} also applies in this setting, from which the result follows.
\end{proof}

\subsection{The resampling argument}
\label{subsec::resampling_argument}

We are now going to work towards completing the proof of Theorem~\ref{thm::map_determines_embedding}.  The proof consists of two main steps.

\begin{enumerate}
\item Show that a $\sqrt{8/3}$-quantum cone is a.s.\ determined by the quantum surfaces which correspond to a finite number of metric hulls (which may overlap), their complement, and the mm-space structure of the overall surface (Lemma~\ref{lem::balls_embedding_determined}).
\item Show that, on the event $\CA_{r,s}$ introduced in Section~\ref{subsec::quantum_cone_stable}, the metric hull $\qhull{0}{s}$ is a.s.\ determined by the mm-space structure up to a global rotation (Lemma~\ref{lem::stable_is_determined}).
\end{enumerate}

We then finish by deducing Theorem~\ref{thm::map_determines_embedding} from Lemma~\ref{lem::stable_is_determined}.

\begin{lemma}
\label{lem::balls_embedding_determined}
Suppose that $(\C,h,0,\infty)$, $(\C,\wt{h},0,\infty)$ are smooth canonical descriptions of $\sqrt{8/3}$-quantum cones and let $\Gamma, \wt{\Gamma}$ (resp.\ $\qdist,\qdistT$) be the space-filling curves (resp.\ metrics) associated with the corresponding Brownian plane structures of $h,\wt{h}$, respectively, with time normalized so that $\Gamma(0) = \wt{\Gamma}(0) = 0$.  Suppose that $t_1,\ldots,t_k \in \R$.  Let $X = \C \setminus \cup_{j=1}^k \partial \qhull{\Gamma(t_j)}{\epsilon}$ and $\wt{X} = \C \setminus \cup_{j=1}^k \partial \qhullT{\wt{\Gamma}(t_j)}{\epsilon}$.  Suppose that there exists a conformal transformation $\psi \colon X \to \wt{X}$ such that $\qdistT(\psi(x),\psi(y)) = \qdist(x,y)$ for all $x,y \in X$, $\mu_{\wt{h}}(\psi(A)) = \mu_h(A)$ for all open $A \subseteq X$, and $\psi(z)-z \to 0$ as $z \to \infty$.  Then $\psi$ extends continuously to $\C$ and is given by the identity map.  In particular, $(\C,h,0,\infty)$ and $(\C,\wt{h},0,\infty)$ are a.s.\ equivalent as quantum surfaces.
\end{lemma}
\begin{proof}
We are first going to show that $\psi$ extends to a homeomorphism $\C \to \C$.  To see this, we note that there exists $c_0,c_1,a_0,a_1 > 0$ (random) such that for each compact set $K \subseteq \C$ and all $x,y \in K \setminus X$ we have that
\begin{align*}
	|\psi(x) - \psi(y)|
&\leq c_0 \qdistT(\psi(x),\psi(y))^{a_0} \quad\text{(Theorem~\ref{thm::metric_completion})}\\
&= c_0 \qdist(x,y)^{a_0} \quad\text{($\psi$ is an isometry)}\\
&\leq c_1 |x-y|^{a_1} \quad\text{(Theorem~\ref{thm::metric_completion})}.
\end{align*}
Therefore $\psi$ is H\"older continuous on $K \setminus X$ hence extends to be H\"older continuous on~$K$.  Since $K \subseteq \C$ was an arbitrary compact set, we therefore have that~$\psi$ is a locally H\"older continuous map $\C \to \C$.  The same argument implies that~$\psi^{-1}$ extends to be a locally H\"older continuous map $\C \to \C$.  Therefore~$\psi$ is a homeomorphism $\C \to \C$ as desired.  Theorem~\ref{thm::finite_unions_removable} combined with \cite[Proposition~2.1]{map_making} implies that $X$ is a.s.\ conformally removable since each $\partial \qhull{\Gamma(t_j)}{\epsilon}$ is a.s.\ the boundary of a simple H\"older domain.  Therefore~$\psi$ is in fact a conformal transformation $\C \to \C$ with $\psi(z) - z \to 0$ as $z \to \infty$.  This implies that~$\psi$ extends to the identity map on~$\C$, as desired.
\end{proof}

\begin{lemma}
\label{lem::inside_outside_independent}
Suppose that $(\C,h,0,\infty)$ is a $\sqrt{8/3}$-quantum cone and $s > 0$.  Let~$\CM$ be the $\sigma$-algebra generated by the mm-space structure $(M,d,\mu)$ (as a random variable taking values in the space of marked mm-spaces using the $\sigma$-algebra from Section~\ref{subsec::mm_spaces}) associated with $h$.  Then the quantum surfaces $(\qhull{0}{s},h)$ and $(\C \setminus \qhull{0}{s},h)$ are conditionally independent given $\CM$.
\end{lemma}
\begin{proof}
Let $\CM_I$ and $\CM_O$ be the $\sigma$-algebras generated by the mm-space structures associated with the quantum surfaces $(\qhull{0}{s},h)$ and $(\C \setminus \qhull{0}{s},h)$, equipped with their internal metric, together with a marked point on $\partial \qhull{0}{s}$.  Then we clearly have that $\CM_I,\CM_O \subseteq \CM$.  We claim that we in fact have that $\CM = \sigma(\CM_I,\CM_O)$.  We first note that~$\CM_I$ and~$\CM_O$ together determine the geodesics from every point $z$ back to $0$ because if $z \in \qhull{0}{s}$ then such a geodesic is contained in $\qhull{0}{s}$ and if $z \notin \qhull{0}{s}$ then such a geodesic can be expressed as a concatenation of a shortest path from~$z$ to $\partial \qhull{0}{s}$ in $\C \setminus \qhull{0}{s}$ and then another shortest path in $\qhull{0}{s}$ from $\partial \qhull{0}{s}$ to $0$.  (The marked point on $\partial \qhull{0}{s}$ is important here because it tells us how to identify the boundaries of the surfaces parameterized by $\qhull{0}{s}$ and by $\C \setminus \partial \qhull{0}{s}$, hence properly concatenate geodesics in $\C \setminus \partial \qhull{0}{s}$ with those inside of $\qhull{0}{s}$.  Namely, we identify them at the marked point and then everywhere else along the boundary according to boundary length.)  As TBP is locally absolutely continuous with respect to TBM \cite{cl2012brownianplane} and in the case of TBM the tree of geodesics of this form determines the entire metric (recall its construction in Section~\ref{subsec::map_plane}), it follows that $\CM_I$ and $\CM_O$ in particular determine the length of paths contained in $\qhull{0}{2s}$.  Therefore~$\CM_I$ and~$\CM_O$ together determine the length of any path because any path can be written as a concatenation of paths which stay outside of~$\qhull{0}{s}$ and which stay in~$\qhull{0}{2s}$.  We conclude that $\CM = \sigma(\CM_I,\CM_O)$, as desired.  Note also that the $\sigma$-algebra~$\CB$ generated by the quantum boundary length of $\partial \qhull{0}{s}$ is $\CM$-measurable.  Indeed, it is shown in the proof of \cite[Proposition~3.31]{map_making} that $\CB \subseteq \CM_I$.  (This is proved by considering geodesics from the boundary of a filled metric ball back to its root with their starting points spaced so that it takes exactly $\epsilon > 0$ units of geodesic distance for geodesics starting from adjacent points to merge.  The boundary length is then given by a constant times the limit as $\epsilon \to 0$ of $\epsilon^{1/2} N_\epsilon$ where $N_\epsilon$ is the number of such paths required to go all of the way around the entire boundary of the filled metric ball.)  One can apply the same argument to see that the boundary length $\partial \qhull{0}{r}$ is determined by $\CM_O$ for any $r > s$ since $\CM_O$ determines the number of geodesics as above required to go all of the way around $\partial \qhull{0}{r}$ provided $\epsilon \in (0,r-s)$.  Since the boundary length process is right-continuous in $r$, it follows that $\CB \subseteq \CM_O$.

Let~$\CF$ be the $\sigma$-algebra generated by the quantum surface $(\C \setminus \qhull{0}{s},h)$ and note that $\CM_O \subseteq \CF$.  Suppose that $A$ is an event which is measurable with respect to the $\sigma$-algebra generated by the quantum surface $(\qhull{0}{s},h)$ and $B \in \CM_I$ is a positive probability event.  By the construction of $\QLE(8/3,0)$, we know that the quantum surface $(\qhull{0}{s},h)$ is conditionally independent of the quantum surface $(\C \setminus \qhull{0}{s},h)$ given its quantum boundary length.  We also know that this quantum boundary length is $\CF$-measurable and that~$\CM_I$ is contained in the $\sigma$-algebra generated by the quantum surface $(\qhull{0}{s},h)$.  Consequently, we have that
\begin{equation}
\label{eqn::cond_ind_equation}
 \p[ A \giv B, \CF] = \frac{\p[ A,B \giv \CF]}{\p[B \giv \CF]} = \frac{\p[ A,B \giv \CB]}{\p[B \giv \CB]} = \p[A \giv B, \CB].
\end{equation}
Since~\eqref{eqn::cond_ind_equation} holds for all such events $A,B$, it follows that for all such events $A$ we have that
\begin{equation}
\label{eqn::cond_ind_equation2}
	\p[ A \giv \CM_I, \CF] = \p[ A \giv \CM_I].
\end{equation}
Since~\eqref{eqn::cond_ind_equation2} holds for all such events $A$, $\CM = \sigma(\CM_I,\CM_O)$, $\CM_I$ and $\CM_O$ are conditionally independent given $\CB$, and $\CB \subseteq \CM_I \cap \CM_O$, it follows that
\begin{equation}
\label{eqn::cond_ind_equation3}
	\p[ A \giv \CM, \CF] = \p[ A \giv \CM].
\end{equation}
for all such events $A$.  It follows from~\eqref{eqn::cond_ind_equation3} that the quantum surface $(\qhull{0}{s},h)$ is conditionally independent of the quantum surface $(\C \setminus \qhull{0}{s},h)$ given $\CM$, as desired.
\end{proof}

\begin{lemma}
\label{lem::stable_is_determined}
Let $(\C,h,0,\infty)$ be a smooth canonical description of a $\sqrt{8/3}$-quantum cone.  Fix $r,s > 0$, and let $\CA_{r,s}$ be the event that $h$ is $(r,s)$-stable.  Let $\CF$ be the $\sigma$-algebra generated by the mm-space structure $(M,d,\mu)$ (as a random variable taking values in the space of marked mm-spaces using the $\sigma$-algebra from Section~\ref{subsec::mm_spaces}) associated with~$h$ and the quantum surface $(\C \setminus \qhull{0}{s},h)$.  Let~$\wt{h}$ be the field which describes the quantum surface generated by starting with~$h$ and then resampling the quantum surface $(\qhull{0}{s},h)$ according to its conditional law given~$\CF$.  Let $\Gamma,\wt{\Gamma}$ be the space-filling paths associated with the Brownian plane structures of $h,\wt{h}$, respectively, with time normalized so that $\Gamma(0) = \wt{\Gamma}(0) = 0$.  Fix $-r^{-1} \leq a < b \leq r^{-1}$ and let $K_{a,b}$ (resp.\ $\wt{K}_{a,b}$) be the complement of the unbounded component of $\C \setminus \Gamma([a,b])$ (resp.\ $\C \setminus \wt{\Gamma}([a,b])$).  We assume that the surface $(\C,\wt{h},0,\infty)$ is embedded so that there exists a unique conformal map $\varphi \colon \C \setminus K_{a,b} \to \C \setminus \wt{K}_{a,b}$ with $|\varphi(z)-z| \to 0$ as $z \to \infty$ and so that $\wt{h} = h \circ \varphi + Q \log|\varphi'|$ in $\C \setminus K_{a,b}$.  On $\CA_{r,s}$ and the event that $K_{a,b} \subseteq \Gamma([-r^{-1},r^{-1}])$, we a.s.\ have that $\varphi(z) = z$ for all $z \in \C \setminus K_{a,b}$.
\end{lemma}

Note that Lemma~\ref{lem::stable_is_determined} implies that $K_{a,b}$ is $\CF$-measurable for all $-r^{-1} \leq a < b \leq r^{-1}$, up to a global rotation, on $\CA_{r,s}$ and the event that $K_{a,b} \subseteq \Gamma([-r^{-1},r^{-1}])$.  Note that for every $\epsilon > 0$, we can find $-r^{-1} < a_0 < a_1 < \cdots < a_k < r^{-1}$ such that $K_{a_{i-1},a_i} \subseteq \Gamma([-r^{-1},r^{-1}])$ and $a_i - a_{i-1} < \epsilon$ for each $1 \leq i \leq k$ and $a_0 + r^{-1}, r^{-1} - a_k < \epsilon$.  By taking a limit as $\epsilon \to 0$, this implies that $\Gamma|_{[-r^{-1},r^{-1}]}$ is $\CF$-measurable on $\CA_{r,s}$.

\begin{proof}[Proof of Lemma~\ref{lem::stable_is_determined}]
Fix $\epsilon > 0$, let $k = 2r^{-1} \epsilon^{-1}$, and assume that $r,\epsilon$ are such that $k$ is an integer. Let  $-r^{-1} \leq t_0 < t_1 < \cdots < t_k \leq r^{-1}$ be $k+1$ equally spaced times.

Let $\CM$ be the $\sigma$-algebra generated by the mm-space structure $(M,d,\mu)$ (as a random variable taking in the space of marked mm-spaces with the $\sigma$-algebra from Section~\ref{subsec::mm_spaces}) associated with~$h$.

Fix $u > 0$ small and let $\xi = 1/4 - u$.  We let $h^0 = h$ (resp.\ $\wt{h}^0 = \wt{h}$) and let $\Gamma^0$ (resp.\ $\wt{\Gamma}^0$) denote the space-filling path associated with TBP structure of $h^0$ (resp.\ $\wt{h}^0$).  Fix $j \geq 1$ and assume that we have defined fields $h^0,\ldots,h^{j-1}$ and $\wt{h}^0,\ldots,\wt{h}^{j-1}$ with associated space-filling paths $\Gamma^0,\ldots,\Gamma^{j-1}$ and $\wt{\Gamma}^0,\ldots,\wt{\Gamma}^{j-1}$.  We then let~$h^j$ be the field which describes the quantum surface which is constructed by starting with~$h^{j-1}$ and then resampling the quantum surface $(\qhullj{\Gamma^{j-1}(t_{j-1})}{\epsilon^{\xi}}{j-1},h^{j-1})$, where $\qhullj{\Gamma^{j-1}(t_{j-1})}{\epsilon^{\xi}}{j-1}$ is the hull of the metric ball of radius~$\epsilon^{\xi}$ centered at~$\Gamma^{j-1}(t_{j-1})$ using the metric associated with~$h^{j-1}$, with respect to the field according to its conditional law given the quantum surface $(\C \setminus \qhullj{\Gamma^{j-1}(t_{j-1})}{\epsilon^{\xi}}{j-1},h^{j-1})$ and $\CM$.  We take $h^j$ so that the embedding of the surface $(\C, h^j,0,\infty)$ is such that if $\varphi_j$ is the conformal map which takes $\C \setminus \qhullj{\Gamma^{j-1}(t_{j-1})}{\epsilon^{\xi}}{j-1}$ to $\C \setminus \qhullj{\Gamma^j(t_{j-1})}{\epsilon^{\xi}}{j}$ then $|\varphi_j(z) - z| \to 0$ as $z \to \infty$ and $h^j = h^{j-1} \circ \varphi_j^{-1} + Q \log |(\varphi_j^{-1})'|$ in $\C \setminus \qhullj{\Gamma^j(t_{j-1})}{\epsilon^{\xi}}{j}$.

Note that transforming from the field $h^0$ to $h^1$ involves cutting out the surface $\qhull{\Gamma(t_0)}{\epsilon^\xi}$ and then gluing in a new surface.  On $\CA_{r,s}$, we have that $\Gamma([-r^{-1},r^{-1}]) \subseteq \qhull{0}{s}$ and therefore $\qhull{\Gamma(t_0)}{\epsilon^{\xi}} \subseteq \qhull{0}{2s}$.  Therefore it follows from the definition of $(r,s)$-stability that the field $h^1$ is $(r,s)$-stable as well.  By induction and the same argument, it follows that all of the fields $h^j$ are $(r,s)$-stable.  We let $\Gamma^j$ be the space-filling path associated with~$h^j$.  We note that $h^j$ is not necessarily a smooth canonical description, however on $\CA_{r,s}$ (by the definition of $(r,s)$-stability) the scaling factor necessary to transform it into a smooth canonical description is between $r^{-1}$ and $r$.

For each $j \geq 1$, we construct $\wt{h}^j$ from $\wt{h}^{j-1}$ in the same way that we constructed $h^j$ from $h^{j-1}$.  We also let $\wt{\Gamma}^j$ and $\qhullTj{x}{r}{j}$ be the space-filling path and filled metric ball associated with the metric of $\wt{h}^j$.  We couple the two resampling procedures together so that $(\qhullj{\Gamma^j(t_j)}{\epsilon^{\xi}}{j},h^j)$ and $(\qhullTj{\wt{\Gamma}^j(t_j)}{\epsilon^{\xi}}{j},\wt{h}^j)$ are equivalent as quantum surfaces for each $j$.  We also define the conformal maps $\wt{\varphi}_j$ in the same way as $\varphi_j$.  As we explained above, since we are working on the event that $h$ is $(r,s)$-stable, it follows from the definition of stability that $h^j$ is $(r,s)$-stable for each $j$.  The same is likewise true for $\wt{h}^j$ for each $j$.

Proposition~\ref{prop::path_chunk_diameter_bound} implies there exists constants $c_0,c_1,a > 0$ such that with probability at least $1-c_0 \exp(-c_1 \epsilon^{-a})$ we have that $\Gamma([-r^{-1},r^{-1}]) \subseteq \cup_{j=1}^k \qhull{\Gamma(t_j)}{\epsilon^\xi}$ as $\xi < 1/4$.  This implies that, with probability at least $1-c_0 \exp(-c_1 \epsilon^{-a})$, the aforementioned procedure resamples the entire quantum surface corresponding to $(K_{a,b},h)$.  (We are using that the statement of Proposition~\ref{prop::path_chunk_diameter_bound} is a statement about the mm-space structure and does not depend on its specific embedding.)  The same is likewise true with $\wt{h}$ in place of $h$.

Let $X = \C \setminus \cup_{j=1}^k \partial \qhullj{\Gamma^k(t_j)}{\epsilon^{\xi}}{k}$ and $\wt{X} = \C \setminus \cup_{j=1}^k \partial \qhullTj{\wt{\Gamma}^k(t_j)}{\epsilon^{\xi}}{k}$.  On the aforementioned event, the resampling procedure yields a conformal map $\varphi \colon X \to \wt{X}$ which satisfies the hypotheses of Lemma~\ref{lem::balls_embedding_determined}.  Therefore $\varphi$ is given by the identity map and the two quantum surfaces described by the fields $h^k$ and $\wt{h}^k$ are equivalent.

To finish proving the lemma, we will show that the maps $\psi = \varphi_k \circ \cdots \circ \varphi_1$ and $\wt{\psi} = \wt{\varphi}_k \circ \cdots \circ \wt{\varphi}_1$ both converge to the identity map as $\epsilon \to 0$.  We will explain the argument in the former case, as the latter argument is analogous.

We note that $h^k$ is a function of the quantum surfaces $(\qhullj{\Gamma^j(t_j)}{\epsilon^{\xi}}{j},h^j)$ for $1 \leq j \leq k$ together with $(\C \setminus \cup_{j=1}^k \qhull{\Gamma(t_j)}{\epsilon^{\xi}} ,h)$.  Indeed, this follows from the same removability argument used to prove Lemma~\ref{lem::balls_embedding_determined}.  Moreover, these surfaces are by construction conditionally independent given $\CF$.  With $a_1$ as in Lemma~\ref{lem::metric_ball_stable_moment_bound}, we assume that $u > 0$ is chosen sufficiently small so that $a_1 \xi > 1$.  Fix $z \in \C$ with $|z|$ sufficiently large so that $|\psi(z)| > 1$.  Note by Lemma~\ref{lem::conformal_map_bound} and since we are working on $\CA_{r,s}$, we are just assuming that $z$ is not in a deterministic ball centered at $0$.  Let $\CG_j$ be the $\sigma$-algebra generated by $(\qhullj{\Gamma^i(t_i)}{\epsilon^\xi}{i},h^i)$ for each $i \neq j$.  Lemma~\ref{lem::variance_bound} implies that
\begin{equation}
\label{eqn::var_map_bound1}
\ex{\var{\psi(z) \giv \CF} \one_{\CA_{r,s}}}
\leq \sum_{j=1}^k \ex{ \var{\psi(z) \giv \CF, \CG_j} \one_{\CA_{r,s}}}.
\end{equation}
Fix $j$ and suppose that $\wh{\psi}$ is a conditionally independent realization of $\psi$ given $\CF$, $\CG_j$.  Then there exists a map $\phi$ which is conformal off $\qhullj{\Gamma^{k}(t_j)}{\epsilon^\xi}{k}$ with $|\phi(w)-w| \to 0$ as $w \to \infty$ such that $\wh{\psi} = \phi \circ \psi$.  Indeed, this follows because the surfaces corresponding to $\psi$ and $\wh{\psi}$ are the same except possibly in $\qhullj{\Gamma^{k}(t_j)}{\epsilon^\xi}{k}$.  Thus we have that $\wh{\psi}(z) - \psi(z) = \phi(\psi(z)) - \psi(z)$.  By Lemma~\ref{lem::conformal_map_bound}, there exists a constant $c_0 > 0$ such that $|\phi(w)-w| \leq c_0 \diam(\qhullj{\Gamma^k(t_j)}{\epsilon^\xi}{k})^4$ for any $|w| > 1$ on $\CA_{r,s}$ as $\qhullj{\Gamma^k(t_j)}{\epsilon^\xi}{k} \subseteq B(0,1/2)$.  Applying this to $w = \psi(z)$, we thus see that $|\wt{\psi}(z) - \psi(z)| \leq c_0 \diam(\qhullj{\Gamma^k(t_j)}{\epsilon^\xi}{k})^4$ (possibly increasing $c_0$).  Combining (and using the representation of the variance as $1/2$ times the expected square difference between independent samples), we have that
\begin{equation}
\label{eqn::var_map_bound2}
\ex{ \var{\psi(z) \giv \CF, \CG_j} \one_{\CA_{r,s}}} \leq c_0 \ex{\diam(\qhullj{\Gamma^k(t_j)}{\epsilon^\xi}{k})^4 \one_{\CA_{r,s}}}.
\end{equation}
Combining~\eqref{eqn::var_map_bound1} and~\eqref{eqn::var_map_bound2} (and possibly increasing $c_0$) we see that
\begin{align*}
\ex{\var{\psi(z) \giv \CF} \one_{\CA_{r,s}}}
&\leq \sum_{j=1}^k c_0 \ex{ \diam(\qhullj{\Gamma^k(t_j)}{\epsilon^{\xi}}{k})^4\one_{\CA_{r,s}} } \\
&\leq c_0 k r^4 \epsilon^{a_1 \xi} \quad\text{(Lemma~\ref{lem::metric_ball_stable_moment_bound})}\\
&= 2 c_0 r^5 \epsilon^{a_1 \xi-1} \to 0 \quad\text{as}\quad \epsilon \to 0 \quad\text{($a_1 \xi > 1$)}.
\end{align*}
The same argument also applies to $\wt{\psi}$.  This implies that $\varphi$ is $\CF$-measurable on the intersection of $\CA_{r,s}$ and the event that $K_{a,b} \subseteq \Gamma([-r^{-1},r^{-1}])$, which implies the result.
\end{proof}

\begin{proof}[Proof of Theorem~\ref{thm::map_determines_embedding}]
Let~$\CF$ be as in Lemma~\ref{lem::stable_is_determined}.  As explained earlier, Lemma~\ref{lem::stable_is_determined} implies that on the event~$\CA_{r,s}$ that~$h$ is $(r,s)$-stable, we have that $\Gamma|_{[-r^{-1},r^{-1}]}$ is $\CF$-measurable.  We note that the intersection of the $\sigma$-algebras generated by the quantum surfaces $(\C \setminus \qhull{0}{s},h)$ for $s > 0$ is trivial.  By applying scaling and using the scale invariance of the law of a $\sqrt{8/3}$-quantum cone, it therefore follows that the probability that~$\Gamma$, hence~$\mu_h$, hence~$h$ \cite{bss2014equivalence} is determined up to a global scaling and rotation by~$\CM$ is at least~$\pr{\CA_{r,s}}$.  The result follows because Proposition~\ref{prop::stable_positive_probability} implies that we can adjust $r, s$ so that~$\pr{\CA_{r,s}}$ is arbitrarily close to~$1$.
\end{proof}

\appendix

\section{Conformal removability}
\label{app::removability}

It was proved by Jones and Smirnov \cite{js2000remove} that the boundary of a H\"older domain is conformally removable.  It is remarked in \cite{js2000remove} that it is not known in general under what conditions finite unions of conformally removable sets are conformally removable.  We will now prove a slight extension of their result to give that a finite union of H\"older domain boundaries is conformally removable provided each of the boundaries is simple.

We have included the proof of this result here for completeness but emphasize that what follows is nearly identical to \cite[Section~2]{js2000remove}.

\begin{theorem}
\label{thm::finite_unions_removable}
Suppose that $D_1,\ldots,D_k \subseteq \C$ are H\"older domains such that each $\partial D_j$ can be parameterized by a simple curve.  Then $X = \cup_{j=1}^k \partial D_j$ is conformally removable.  That is, if $U \subseteq \C$ is a domain and $\varphi \colon U \to V$ is a homeomorphism which is conformal on $U \setminus X$ then $\varphi$ is conformal on all of $U$.
\end{theorem}

The assumption that each $\partial D_j$ is parameterized by a simple curve will be important in the proof of Theorem~\ref{thm::finite_unions_removable} because it implies that for each $\epsilon > 0$ there exists $\delta > 0$ such that if we have $z,w \in \partial D_j$ distinct with $|z-w| \leq \delta$ then there exists a path contained in the interior of $D_j$ except at its endpoints which connects $z$ to $w$ and has diameter at most $\epsilon$.

Suppose that $D \subseteq \C$ is a Jordan domain and let $\varphi \colon \D \to D$ be a conformal transformation.  Let $z = \varphi(0)$ and let $\Gamma$ be the family of paths which correspond to the images under $\varphi$ of the line segments $[0,e^{i \theta}]$ for $\theta \in [0,2\pi]$.  Let $\CW$ a Whitney cube decomposition of $D$.  For each $Q \in \CW$, we define the \emph{shadow} $\SH(Q)$ of $Q$ to be the set of points in $\partial D$ which are the endpoint of an arc in $\Gamma$ which passes through $Q$.  We also let $s(Q) = \diam(\SH(Q))$ and let $L(Q)$ and $|Q|$ respectively denote the side length and area of $Q$.  We shall be interested in domains $D$ which satisfy the condition:
\begin{equation}
\label{eqn::shadows_summable}
\sum_{Q \in \CW} s(Q)^2 < \infty.
\end{equation}
It is explained just after the statement of \cite[Corollary~2]{js2000remove} that a H\"older domain satisfies~\eqref{eqn::shadows_summable}.

We say that a function $f$ on a domain $U \subseteq \C$ is ACL (absolutely continuous on lines) if it is the case that $f$ is absolutely continuous on the intersection with $U$ of almost every line which is parallel either to the vertical or the horizontal coordinate axis.  It is not difficult to check that any function which is conformal almost everywhere and ACL is in fact conformal everywhere.

The main input into the proof of Theorem~\ref{thm::finite_unions_removable} is the following proposition, which is nearly identical to \cite[Proposition~1]{js2000remove}.  Before we give the statement, we recall that for $p \geq 1$ the space $W^{1,p}$ consists of those functions $f$ which are absolutely continuous with first derivative in $L^p$.

\begin{proposition}
\label{prop::acl_up_to_the_boundary}
Suppose that $D \subseteq \C$ is a Jordan domain.  Let $\Gamma$, $\CW$, $\SH(Q)$ and~$s(Q)$ for $Q \in \CW$ be as just above.  Assume that~\eqref{eqn::shadows_summable} holds.  Let~$U$ be any bounded domain such that $\partial D \cap U \neq \emptyset$.  Let~$\gamma$ be a curve which parameterizes~$\partial D$, fix~$t$ so that $\gamma(t) \in U$ and let $t_1$ (resp.\ $t_2$) be the supremum (resp.\ infimum) of times~$s$ before (resp.\ after) $t$ that $\gamma(s) \notin U$.  Let $K = \gamma([t_1,t_2])$.  Then any continuous function $f$ which belongs to $W^{1,2}$ on $U \setminus K$ also belongs to $W^{1,2}$ on $U$.  In particular, such a function $f$ is ACL.
\end{proposition}
\begin{proof}
Fix a direction $v$ and let $\partial_v f$ denote the directional derivative of $f$ in the sense of distributions.  Let $\iint_U$ denote the double integral, where we will first integrate along a line which is parallel to $v$ and then over all such lines.  We will aim to show that
\begin{equation}
\label{eqn::int_over_u}
\iint_U |\partial_v f| = \iint_{U \setminus K} |\partial_v f|
\end{equation}
The identity~\eqref{eqn::int_over_u} means that on Lebesgue almost every line $\ell$ which is parallel to $v$ the total variation of $f$ is equal to $\int_{\ell \cap U \setminus K} |\partial_v f|$.  Since $f \in W^{1,2}(U \setminus K) \subseteq W^{1,1}(U \setminus K)$, by Fubini's theorem this implies that $\partial_v f$ restricted to almost every line $\ell$ parallel to $v$ is in fact an integrable function.  By taking all possible directions $v$, we deduce that $f$ is ACL, thus proving the proposition.

Let $\ell_1,\ell_2$ be the two lines which are parallel to $v$ and pass through $\gamma(t_1)$ and $\gamma(t_2)$.  Fix $\epsilon > 0$ and let $\ell$ be a line which is also parallel to $v$ and with distance at least $\epsilon$ from both $\ell_1$ and $\ell_2$.  We will momentarily integrate over all such lines and apply Fubini's theorem at the end.

We denote the total variation of $f$ on $\ell \cap U$ by $\int_{\ell \cap U} |\partial_v f|$, and note that it can be arbitrarily closely approximated by expressions of the form 
\begin{equation}
\label{eqn::total_var_approx}
\sum_j |f(x_j) - f(y_j)| + \int_{\ell \cap U \setminus \cup_j [x_j,y_j]} |\partial_v f|
\end{equation}
where the pairwise disjoint intervals $[x_j,y_j]$ cover $\ell \cap K$ with $x_j,y_j \in \ell \cap K$.

Let $K_\epsilon$ consist of those $z$ in $K$ which have distance at least $\epsilon$ from $\ell_1$ and $\ell_2$ and let $\Delta_0 = \tfrac{1}{100} \big( \epsilon \wedge \dist(K_\epsilon,\partial U) \big)$.  It then follows that any Whitney cube with distance at least $\epsilon$ from $\ell_1$ and $\ell_2$ and with distance to $K$ at most $\Delta_0$ will be contained in $U$.  By condition~\eqref{eqn::shadows_summable}, as the Whitney cubes get smaller the diameters of their shadows tend to zero.  Hence we can choose such a small size $\Delta \leq \Delta_0$ such that no shadow of a Whitney cube of this or smaller size intersects more than one interval $[x_j,y_j]$, all such cubes are contained in $U$, and the shadows of the cubes of this size cover $K_\epsilon$.

Fix one such interval $[x_j,y_j]$.  Since there are only finitely many cubes of size $\Delta$, and the set $[x_j,y_j] \cap K$ is covered by their shadows (which are compact sets), one can cut $[x_j,y_j]$ into finitely many intervals $[u_i,u_{i+1}]$ so that $u_0 = x_j$ and $u_n = y_j$.  By compactness, we can do this in such a way so that, for every $i$ either $(u_i,u_{i+1}) \subseteq U \setminus K$ or $u_i$ and $u_{i+1}$ belong to the same shadow $\SH(Q_i)$, and there are curves from $\Gamma$, joining~$u_i$ and~$u_{i+1}$ to~$Q_i$, that do not intersect cubes of larger or equal size.

In the first case, we have that
\begin{equation}
\label{eqn::f_bound_deriv}
|f(u_i) - f(u_{i+1})| \leq \int_{[u_i,u_{i+1}]} |\partial_v f|.
\end{equation}
In the latter case, we can join~$u_i$ and~$u_{i+1}$ by a curve~$\gamma_i$ which follows one $\Gamma$-curve from~$u_i$ to the cube~$Q_i$ and then switches to another $\Gamma$-curve from~$Q_i$ to~$u_{i+1}$.

For an integrable function $g$ and $Q \in \CW$, we denote by
\[ g(Q) = \frac{1}{|Q|} \int_Q g\]
its mean value on $Q$.  For any two adjacent $Q, Q' \in \CW$ (i.e.\ such that they have the same side length and share a face, or one of them has twice the side length of the other and they share a face of the smaller one) contained in $U$ we observe that there exists a constant $c_0 > 0$ such that
\begin{equation}
\label{eqn::f_adj_cube_bound}
| f(Q) - f(Q')| \leq c_0 ( |D f|(Q) L(Q) + |D f|(Q') L(Q')),
\end{equation}
where $Df$ is the vector of partial derivatives of $f$ and $|Df|$ is its $L^1$ norm.  Indeed, this can be seen by letting $x$ (resp.\ $x'$) be the lower left corner of $Q$ (resp.\ $Q'$) and then considering the bound
\begin{align*}
   |f(Q) - f(Q')|
&\leq \int_{[0,1]^2} |f(L(Q) u + x) - f(L(Q')u+x')| du.
\end{align*}
Taking the Whitney cubes intersecting the curve $\gamma_i$ and excluding some of them one can choose a bi-infinite sequence of cubes such that its tails converge to $u_i$ and $u_{i+1}$ correspondingly, and any two consecutive cubes are adjacent.  Applying~\eqref{eqn::f_adj_cube_bound} (and possibly increasing the value of $c_0$), one obtains
\begin{equation}
\label{eqn::f_bound_cubes}
|f(u_i) - f(u_{i+1})| \leq c_0 \sum_{Q \cap \gamma_i \neq \emptyset} |D f|(Q) L(Q),
\end{equation}
where the sum is taken over all Whitney cubes intersecting $\gamma_i$ (even at a single point).  All the cubes in the estimate~\eqref{eqn::f_bound_cubes} have size at most $\Delta$, and by the choice of $\Delta$ they belong to $U$.

Adding up the estimates~\eqref{eqn::f_bound_deriv},~\eqref{eqn::f_adj_cube_bound}, and~\eqref{eqn::f_bound_cubes} for all $i$, we obtain
\begin{align*}
        |f(x_j) - f(y_j)|
&\leq  \sum_i |f(u_i) - f(u_{i+1})|\\
&\leq  \sum_{[ u_i,u_{i+1}] \subseteq U \setminus K} \int_{[u_i,u_{i+1}]} |\partial_v f| + c_0 \sum_{[u_i,u_{i+1}] \nsubseteq U \setminus K} \sum_{Q \cap \gamma_i \neq \emptyset} |D f|(Q) L(Q).
\end{align*}

The first term is bounded from above by $\int_{[x_j,y_j] \setminus K} |\partial_v f|$.  Note that all Whitney cubes that arise in the second term have one of the points $u_i$ in their shadow and are of size at most $\Delta$.  As the following reasoning shows, for the purpose of estimating $|f(x_j) - f(y_j)|$, we can assume that no cube appears twice in the sums.  In fact, if there is a Whitney cube $Q$, entered by two curves $\gamma_k$ and $\gamma_l$, $k < l$, then we can make a new curve out of them, connecting $u_k$ directly to $u_{l+1}$, and thus improving the estimate above, writing
\[ |f(x_j) - f(y_j)| \leq \sum_{i < k} |f(u_i) - f(u_{i+1})| + |f(u_k) - f(u_{l+1})| + \sum_{i > l} | f(u_i) - f(u_{i+1})|\]
and obtaining fewer cubes in the resulting estimate.  If necessary, we can repeat this procedure several, and thus assume that no Whitney cube is entered by two different curves.

Therefore we can rewrite our estimate as
\begin{equation}
\label{eqn::better_estimate}
|f(x_j) - f(y_j)| \leq \int_{[x_j,y_j] \setminus K} |\partial_v f| + c_0 \sum_{\SH(Q) \cap [x_j,y_j] \neq \emptyset} |D f|(Q) L(Q).
\end{equation}
Recalling that by the choice of~$\Delta$, no shadow of a cube of that or smaller size intersects more than one interval $[x_j,y_j]$, we conclude that every cube $Q$ in the estimate~\eqref{eqn::better_estimate}  appears for at most one $j$ and has shadow intersecting~$\ell$.  Now, summing~\eqref{eqn::better_estimate} over all~$j$, we obtain the following estimate of the expression~\eqref{eqn::total_var_approx}
\begin{equation}
\label{eqn::abc}
\int_{\ell \cap U} |\partial_v f| \leq c_0 \sum_{\SH(Q) \cap \ell \neq \emptyset} |D f|(Q) L(Q) + \int_{\ell \cap U \setminus K} |\partial_v f|.
\end{equation}
Moreover, only Whitney cubes of size at most $\Delta$ (which we can choose to be arbitrarily small) are included in the latter estimate.

Notice that a Whitney cube $Q$ participates in the estimate only if the line $\ell$ intersects its shadow, and the measure of the set of such lines is at most $s(Q)$.  Let $U_\epsilon$ be the set of points in $U$ which have distance at least $\epsilon$ from both $\ell_1$ and $\ell_2$.  Integrating~\eqref{eqn::abc} over all lines $\ell$, parallel to the direction $v$ and with distance at least $\epsilon$ from $\ell_1$ and $\ell_2$, (here $\mu$ denotes the transversal measure on those lines), and applying Fubini's theorem we obtain
\begin{align}
      \iint_{U_\epsilon} |\partial_v f|
& :=  \int \left( \int_{\ell \cap U_\epsilon} |\partial_v f| \right) d\mu(\ell) \notag\\
&\leq \int\left( c_0 \sum_{\SH(Q) \cap \ell \neq \emptyset} |D f|(Q) L(Q) + \int_{\ell \cap U \setminus K} |\partial_v f| \right) d\mu(\ell) \quad\text{(by~\eqref{eqn::abc})}\notag\\
&= \int\left( c_0 \sum |D f|(Q) L(Q) \one_{\SH(Q) \cap \ell \neq \emptyset} + \int_{\ell \cap U \setminus K} |\partial_v f| \right) d\mu(\ell) \notag\\
&\leq c_0 \left( \sum |D f|(Q) L(Q) s(Q) \right) + \iint_{U_\epsilon \setminus K} |\partial_v f|. \label{eqn::variation_estimate}
\end{align}
In the final inequality, we used that $\int \one_{\SH(Q) \cap \ell \neq \emptyset} d\mu(\ell) \leq s(Q)$.  We will now argue that the first sum in~\eqref{eqn::variation_estimate} is finite.  Indeed, by the Cauchy-Schwarz inequality and using that $|Q| = (L(Q))^2$, we have that
\begin{align}
       \sum |D f|(Q) L(Q) s(Q)
&\leq \left( \sum |D f|(Q)^2 |Q| \right)^{1/2} \left( \sum\big( s(Q) / L(Q) \big)^2 |Q| \right)^{1/2} \notag\\
&\leq \left( \sum |D f|(Q)^2 |Q| \right)^{1/2} \left( \sum s(Q)^2 \right)^{1/2}. \label{eqn::cs_variation_bound}
\end{align}
The first sum on the right hand side of~\eqref{eqn::cs_variation_bound} is bounded by the $W^{1,2}(U \setminus K)$ norm of~$f$, hence finite.  Condition~\eqref{eqn::shadows_summable} implies that the second sum on the right hand side of~\eqref{eqn::cs_variation_bound} is finite.  As before, we can assume that only Whitney cubes of a small size $\Delta$ (which we are free to choose) participate in this sum.  Thus by taking a limit as $\Delta \to 0$ (but with $\epsilon > 0$ fixed), we say that the second sum in the right side of~\eqref{eqn::cs_variation_bound} tends to~$0$.  We therefore arrive at the bound
\begin{equation}
\label{eqn::nearly_final_bound}
\iint_{U_\epsilon} |\partial_v f| \leq \iint_{U_\epsilon \setminus K} |\partial_v f|.
\end{equation}
Sending $\epsilon \to 0$, using that $K$ is between $\ell_1$ and $\ell_2$, and applying the monotone convergence theorem,~\eqref{eqn::nearly_final_bound} we see that
\begin{equation}
\label{eqn::var_inequality}	
\iint_{U} |\partial_v f| \leq \iint_{U \setminus K} |\partial_v f|.
\end{equation}
Clearly, both sides of~\eqref{eqn::var_inequality} are therefore equal, thus proving the desired equality~\eqref{eqn::int_over_u}, and hence the proposition.
\end{proof}

\begin{proof}[Proof of Theorem~\ref{thm::finite_unions_removable}]
Suppose that $D_1,\ldots,D_k \subseteq \C$ are H\"older domains such that each $X_j = \partial D_j$ may be parameterized by a simple curve.  Let $X = \cup_{j=1}^k X_j$.  Suppose that $W \subseteq \C$ is an open set containing $X$ and that $f$ is a homeomorphism on $W$ which is conformal on $W \setminus X$.

Pick $z \in X_1 \setminus \cup_{j=2}^k X_j$.  We will show that~$f$ is conformal in neighborhood of~$z$.  Let $\gamma_1 \colon [0,1] \to X_1$ be a simple curve which parameterizes $X_1$ (except $\gamma_1(0) = \gamma_1(1)$).  Assume that $\gamma_1(0) \neq z$ and let $t \in (0,1)$ be the unique value so that $\gamma_1(t) = z$.    Fix $\epsilon > 0$ so that $B(z,\epsilon)$ intersects $X_1$ but not $X_j$ for $j \neq 1$.  Assume also that $B(z,\epsilon)$ does not contain $\gamma_1(0)$.  Let $t_1$ be the largest time~$s$ before~$t$ so that $\gamma_1(s)$ is not in $B(z,\epsilon)$ and let~$t_2$ be the smallest time~$s$ after~$t$ so that $\gamma_1(s)$ is not in $B(z,\epsilon)$.  Let $U = B(z,\epsilon) \setminus (\gamma_1([0,t_1]) \cup \gamma_1([t_2,1]))$.  Then Proposition~\ref{prop::acl_up_to_the_boundary} implies that $f$ is ACL in~$U$.  Since~$f$ is also conformal almost everywhere, this implies that~$f$ is conformal in~$U$.  In particular, $f$ is conformal in a neighborhood of~$z$.  Since $z \in X_1 \setminus \cup_{j=2}^k X_j$ with $z \neq \gamma_1(0)$ was arbitrary, this implies that~$f$ is conformal in $D \setminus \cup_{j=2}^k X_j$ except possibly at~$\gamma_1(0)$.  However, by using a different choice of parameterization of~$X_1$, we see that~$f$ is conformal at~$\gamma_1(0)$ as well.  The result therefore follows by induction on~$k$.
\end{proof}

\bibliographystyle{hmralphaabbrv}
\addcontentsline{toc}{section}{References}
\bibliography{sle_kappa_rho}

\bigskip

\filbreak
\begingroup
\small
\parindent=0pt

\bigskip
\vtop{
\hsize=5.3in
Statistical Laboratory, DPMMS\\
University of Cambridge\\
Cambridge, UK}

\bigskip
\vtop{
\hsize=5.3in
Department of Mathematics\\
Massachusetts Institute of Technology\\
Cambridge, MA, USA } \endgroup \filbreak

\end{document}